\documentclass[11pt,reqno]{amsart}
\usepackage{amssymb,latexsym}
\usepackage{graphicx}
\usepackage{fancyhdr}
\usepackage{tikz}
\usepackage[all]{xy}
\usepackage{enumerate}
\theoremstyle{plain}
\numberwithin{equation}{section}
\newtheorem{thm}{Theorem}[section]
\newtheorem{theorem}[thm]{Theorem}

\newtheorem{corollary}[thm]{Corollary}

\newtheorem{lemma}[thm]{Lemma}

\newtheorem{definition}[thm]{Definition}
\newtheorem{proposition}[thm]{Proposition}
\newtheorem*{theorem*}{Theorem}

\begin{document}
\setcounter{page}{1}
\title[Extensions of $3$-Lie superalgebras]{Cohomology, superderivations, and abelian extensions of $3$-Lie superalgebras}

\author[Nandi]{Nupur Nandi}
\address{National Institute of Technology Rourkela \\
	Odisha-769008 \\
	India}
\email{nupurnandi999@gmail.com}

\author[Padhan]{Rudra Narayan Padhan}
\address{Centre for Data Science, Institute of Technical Education and Research  \\
	Siksha `O' Anusandhan (A Deemed to be University)\\
	Bhubaneswar-751030 \\
	Odisha, India}
\email{rudra.padhan6@gmail.com, rudranarayanpadhan@soa.ac.in}


\maketitle
\begin{abstract}
	The main object of study of this paper is the notion of 3-Lie superalgebras with superderivations. We consider a representation $(\Phi,\mathcal{P})$ of a $3$-Lie superalgebra $\mathcal{Q}$ on $\mathcal{P}$  and construct first-order cohomologies by using superderivations of $\mathcal{P},\mathcal{Q}$ which induces a Lie superalgebra $\mathcal{T}_{\Phi}$ and its representation $\Psi$. Then we consider abelian extensions of $3$-Lie superalgebras of the form $0\rightarrow \mathcal{P}\hookrightarrow \mathcal{L}\rightarrow \mathcal{Q}\rightarrow 0$ with $[\mathcal{P},\mathcal{P},\mathcal{L}]=0$ and construct an obstruction class to extensibility of a compatible pair of superderivation. Moreover we prove that a pair of superderivation is extensible if and only if its obstruction class is trivial under some suitable conditions.
\end{abstract}

\section{Introduction}
 Filippov introduced $n$-Lie algebras in 1985 \cite{filippov1985}. $n$-Lie algebras, in particular $3$-Lie algebras are important in mathematical physics. Lie superalgebras are the $\mathbb{Z}_2$-graded Lie algebras which was introduced by Kac \cite{kac1977}. These are too interesting from a purely mathematical point of view. The notion of 3-Lie superalgebras are generalization of 3-Lie algebras extending to a $\mathbb{Z}_2$-graded case. $n$-Lie superalgebras are more general structures including $n$-Lie algebras and Lie superalgebras whose definition was introduced by Cantarini $et ~al.$ \cite{cantarini2010}.
  \par Derivation algebra is an important topic in Lie algebras which has widespread applications in physics and geometry.  A superderivation of a Lie superalgebra is certain generalization of derivation of a Lie algebra. The structure of superderivation of Lie superalgebras was studied in \cite{nandi2021,wang2016}. Cohomology is an important tool in modern mathematics and theoretical physics, its range of applications contain algebra and topology as well as the theory of smooth manifolds or holomorphic functions. The cohomology of Lie algebras was defined by Chevalley $et~ al.$ \cite{chevalley1948}. Leites introduced cohomology of Lie superalgebras and extended some of the basic structures and results of classical theories to Lie superalgebras \cite{leites1975}. Further cohomology for $n$-Lie superalgebras was discussed in \cite{ma2014}. 
  \par Recently Tang $et ~al.$ studied a Lie algebra with a derivation from the cohomological point of view and construct a cohomology theory that controls, among other things, simultaneous deformations of a Lie algebra with a derivation \cite{tang2019}. These results have been extended to associative algebras \cite{das2020}, Leibniz algebras \cite{das2021}, $3$-Lie colour algebras \cite{zhang2014}, 3-Lie algebras \cite{xu2018}, Lie triple systems \cite{wu2022}, and $n$-Lie algebras \cite{sun2021}. Generalized representations of 3-Lie algebras and 3-Lie superalgebras was introduced in \cite{zhu2017,zhu2020}. Zhao $et~ al.$ studied a representation of a Lie superalgebra with a superderivation pair and its corresponding cohomologies \cite{zhao2021}.
\par The aim of this paper is to generalize the results of Xu \cite{xu2018} to $3$-Lie superalgebra case. First we take a representation $(\Phi,\mathcal{P})$ of a $3$-Lie superalgebra $\mathcal{Q}$ on $\mathcal{P}$ and construct $2$-cocycles by using superderivations of $\mathcal{P},\mathcal{Q}$ and hence first-order cohomologies. This construction develops a Lie superalgebra $\mathcal{T}_{\Phi}$ by the representation $\Phi$ and the space $\mathbb{H}^1(\mathcal{Q};\mathcal{P})$ of first-order cohomology class gives a representation $\Psi$ of the Lie superalgebra $\mathcal{T}_{\Phi}$. Further we consider representation of $3$-Lie superalgebras given by abelian extensions of $3$-Lie superalgebras of the form $0\rightarrow \mathcal{P}\hookrightarrow \mathcal{L}\rightarrow \mathcal{Q}\rightarrow 0$ with $[\mathcal{P},\mathcal{P},\mathcal{L}]=0$ and construct an obstruction class to extensibility of a compatible pair of superderivations of $\mathcal{P},\mathcal{Q}$ to those of $\mathcal{L}$.

\section{Preliminaries}
In this section, we recall representations and cohomologies of $3$-Lie superalgebras and their relations to abelian extensions of $3$-Lie superalgebras of the form $0\rightarrow \mathcal{P}\hookrightarrow \mathcal{L}\xrightarrow{\pi} \mathcal{Q}\rightarrow 0$ with $[\mathcal{P},\mathcal{P},\mathcal{L}]=0$. We show that $\mathbb{H}^1(\mathcal{Q};\mathcal{P})=0$ for such extensions implies split property. 
\par Let $\mathbb{Z}_2 =\{\overline{0},\overline{1}\}$ be the field of two elements. Throughout the paper, we denote $\mathbb{F}$ as a field of characteristic zero. A superspace is a $\mathbb{Z}_2$-graded vector space $\mathcal{V}=\mathcal{V}_{\overline{0}}\oplus \mathcal{V}_{\overline{1}}$. A $sub  superspace$ is a $\mathbb{Z}_2$-graded vector space which is closed under bracket operation. Non-zero elements of $\mathcal{V}_{\overline{0}}\cup \mathcal{V}_{\overline{1}}$ are said to be $homogeneous$ and whenever the degree function occurs in a formula, the corresponding elements are supposed to be homogeneous.
A $superalgebra$ is a superspace $\mathcal{G}=\mathcal{G}_{\overline{0}}\oplus \mathcal{G}_{\overline{1}}$ endowed with an algebra structure such that $\mathcal{G}_\alpha \mathcal{G}_\beta \subseteq \mathcal{G}_{\alpha + \beta}$ for $\alpha , \beta \in \mathbb{Z}_2$.  
\begin{definition}\label{d21}
	A 3-$Lie ~superalgebra$ is a $\mathbb{Z}_2$-graded vector space $\mathcal{G}=\mathcal{G}_{\overline{0}}\oplus \mathcal{G}_{\overline{1}}$ equipped with a trilinear map $[.,.,.]:\wedge^3\mathcal{G}\rightarrow \mathcal{G}$ satisfying
	\begin{enumerate}
		\item $|[x_1,x_2,x_3]|=|x_1|+|x_2|+|x_3|$,
		\item  $[x_1,x_2,x_3]=-(-1)^{|x_1||x_2|}[x_2,x_1,x_3]=-(-1)^{|x_2||x_3|}[x_1,x_3,x_2]$,
		\item $[x_1,x_2,[x_3,x_4,x_5]]=[[x_1,x_2,x_3],x_4,x_5]+(-1)^{|x_3|(x_1|+|x_2|)}[x_3,[x_1,x_2,x_4],x_5]\\
		+(-1)^{(|x_1|+|x_2|)(|x_3|+|x_4|)}[x_3,x_4,[x_1,x_2,x_5]]$,
	\end{enumerate}
for $x_1,x_2,x_3,x_4,x_5\in \mathcal{G} $ and $|x_i|$ is the degree of homogeneous element $x_i$.
\end{definition}
\par A subsuperspace $\mathcal{N}$ of a $3$-Lie superalgebra $\mathcal{G}$ is said to be a $Lie~ subsuperalgebra$ if it is closed under the superbracket. If $\mathcal{G}$ and $\mathcal{M}$ are 3-Lie superalgebras then a $3$-$Lie~ superalgebra$ $homomorphism$ $\theta:\mathcal{G}\rightarrow \mathcal{M}$ is an even linear map satisfying $\theta([x,y,z])=[\theta(x),\theta(y),\theta(z)]$ for $x,y,z\in \mathcal{G}$.

\begin{definition}\label{d22}
	A superderivation of a 3-$Lie ~superalgebra$ $\mathcal{G}$ is a linear map $\mathcal{D}:\mathcal{G}\rightarrow \mathcal{G}$ of degree $\beta$ satisfying:
	$$\mathcal{D}([x,y,z])=[\mathcal{D}(x),y,z]+(-1)^{|\beta||x|}[x,\mathcal{D}(y),z]+(-1)^{|\beta|(|x|+|y|)}[x,y,\mathcal{D}(z)],$$ for $x,y,z\in \mathcal{G}.$
	\end{definition}
\noindent We denote $Der(\mathcal{G})$ as the space of superderivations of $\mathcal{G}$. Define an even skew-supersymmetric bilinear map $ad:\wedge^2 \mathcal{G}\rightarrow gl(\mathcal{G})$ by $$ad(x_1,x_2)x_3=[x_1,x_2,x_3],$$ for $x_1,x_2,x_3\in \mathcal{G}$.

\begin{definition}\label{d23}
A representation of a $3$-Lie superalgebra $(\mathcal{G},[.,.,.])$ on a superspace $\mathcal{V}$ is a bilinear map $\Phi:\wedge ^2 \mathcal{G}\rightarrow gl(\mathcal{V})$ such that the following equalities hold:
\begin{enumerate}
	\item $|\Phi(x_1,x_2)|=|x_1|+|x_2|$,
	\item $\Phi(x_1,x_2)=-(-1)^{|x_1||x_2|}\Phi(x_2,x_1)$,
	\item $\Phi(x_1,x_2) \Phi(x_3,x_4)=\Phi([x_1,x_2,x_3],x_4)+(-1)^{|x_3|(|x_1|+|x_2|)}\Phi(x_3,[x_1,x_2,x_4])\\
	+(-1)^{(|x_1|+|x_2|)(|x_3|+|x_4|)}\Phi(x_3,x_4)\Phi(x_1,x_2)$,
	\item $\Phi(x_1,[x_2,x_3,x_4])=(-1)^{(|x_1|+|x_2|)(|x_3|+|x_4|)}\Phi(x_3,x_4)\Phi(x_1,x_2)\\-(-1)^{|x_1|(|x_2|+|x_4|)+|x_3||x_4|}\Phi(x_2,x_4)\Phi(x_1,x_3)+(-1)^{|x_1|(|x_2|+|x_3|)}\Phi(x_2,x_3)\Phi(x_1,x_4)$,
\end{enumerate}	for $x_1,x_2,x_3,x_4\in \mathcal{G}.$
\end{definition}
\noindent We denote a representation of $\mathcal{G}$ on a superspace $\mathcal{V}$ by $(\Phi;\mathcal{V})$.
\par Now onwards we always assume that $(\mathcal{G},[.,.,.])$ is a $3$-Lie superalgebra and we shall write, for any $X=x_1\wedge x_2 \in \wedge ^2 \mathcal{G}$, $x_3\in \mathcal{G}$,
\begin{equation}\label{e21}
	[X,x_3]:=[x_1,x_2,x_3]\in \mathcal{G}.
\end{equation}
We shall use the following bilinear operation $[.,.,.]_\mathbb{F}$ on $\wedge ^2 \mathcal{G}$ given by 
\begin{equation}\label{e22}
	[X,Y]_\mathbb{F}=[X,y_1]\wedge y_2 +(-1)^{|y_1||X|} y _1 \wedge[X,y_2] \in \mathcal{G},
\end{equation}
for $X=x_1\wedge x_2,~Y=y_1\wedge y_2$, and $|X|=|x_1|+|x_2|$. One can see that $\wedge ^2 \mathcal{G}$ is a Leibniz superalgebra with respect to $[.,.]_\mathbb{F}$.
\par Let $(\Phi;\mathcal{V})$ be a representation of $\mathcal{G}$. Cohomology groups of $\mathcal{G}$ with coefficients in $\mathcal{V}$ are defined as in \cite{ma2014}. At first, the space $C^{p-1}(\mathcal{G};\mathcal{V})$ of $p$-cochains is the set of multilinear maps of the form 
\begin{equation}\label{e23}
	f:\underbrace{\wedge ^2\mathcal{G} \otimes \wedge ^2\mathcal{G}\otimes \dots \otimes \wedge ^2 \mathcal{G}}_{p-1} \otimes \mathcal{G}\rightarrow \mathcal{V},
\end{equation}
while the coboundary operator $\delta_{\Phi}:C^{p-1}(\mathcal{G};\mathcal{V})\rightarrow C^{p}(\mathcal{G};\mathcal{V})$ is given by
\begin{equation}\label{e24}
	\begin{split}
		&(\delta_{\Phi}f)(X_1,X_2,\dots,X_p,z)\\
		&=\sum_{1\leq j<k\leq p}(-1)^j (-1)^{|X_j|(|X_{j+1}|+\dots+|X_{k-1}|)}f(X_1,\dots,\hat{X_j},\dots,X_{k-1},[x_j^1,x_j^2,x_k^1]\wedge x_k^2,\\&~~~~~~X_{k+1},\dots,X_p,x)+\sum_{1\leq j<k\leq p}(-1)^j (-1)^{|X_j|(|X_{j+1}|+\dots+|X_{k-1}|)+|x_k^1||X_j|}f(X_1,\dots,\hat{X_j},\dots,X_{k-1},\\&~~~~~~x_k^1\wedge[x_j^1,x_j^2,x_k^2],X_{k+1},\dots,X_p,x)+\sum_{j=1}^{p}(-1)^j(-1)^{|X_j|(|X_{j+1}|+\dots+|X_p|)}f(X_1,\dots,\hat{X_j},\dots,\\& X_p, [X_j,x])+\sum_{j=1}^{p}(-1)^{j+1}(-1)^{|X_j|(|f|+|X_1|+\dots+|X_{j-1}|)}\Phi(X_j)f(X_1,\dots,\hat{X_j},\dots, X_p, x)\\
		&+(-1)^{p+1}(-1)^{(|x_p^2|+|x|)(|f|+|X_1|+\dots+|X_{p-1}|+|x_p^1|)}\Phi(x_p^2,x)f(X_1,\dots, X_{p-1},x_p^1)\\
		&+(-1)^{p+1}(-1)^{(|x_p^1|+|x|)(|f|+|X_1|+\dots+|X_{p-1}|)+|X_p||x|}\Phi(x,x_p^1)f(X_1,\dots, X_{p-1},x_p^2),
	\end{split}
\end{equation}
for $X_i=x_i\wedge y_i\in \wedge ^2\mathcal{G}$ and $z\in \mathcal{G}$. The $p$$^{th}$ cohomology group is $\mathbb{H}^p(\mathcal{G};\mathcal{V})=\mathbb{Z}^p(\mathcal{G};\mathcal{V})/\mathbb{B}^p(\mathcal{G};\mathcal{V}),$ where $\mathbb{Z}^p(\mathcal{G};\mathcal{V})$ (respectively, $\mathbb{B}^p(\mathcal{G};\mathcal{V}))$ is the space of $(p+1)$-cocycles(respectively, $(p+1)$-coboundaries). We denote $(p+1)$-cocycles of even degree as $(\mathbb{Z}^p(\mathcal{G};\mathcal{V}))_{\overline{0}}$.\\
\noindent By using Eq \ref{e24}, for $f \in \mathbb{C}^0(\mathcal{G};\mathcal{V})$, $X_1=x_1\wedge x_2\in \wedge ^2 \mathcal{G}$, and $x_3\in \mathcal{G}$, we have 
\begin{equation}\label{e25}
	\begin{split}
		&(\delta_{\Phi} f)(X_1,x_3)=-f([X_1,x_3])+(-1)^{|f|(|x_1|+|x_2|)}\Phi(X_1)f(x_3)\\
		&+(-1)^{(|f|+|x_1|)(|x_2|+|x_3|)}\Phi(x_2,x_3)f(x_1)+(-1)^{|f|(|x_1|+|x_3|)+|x_3|(|x_1|+|x_2|)}\Phi(x_3,x_1)f(x_2)\\
		&=-f([x_1,x_2,x_3])+(-1)^{|f|(|x_1|+|x_2|)}\Phi(x_1,x_2)f(x_3)\\
		&+(-1)^{(|f|+|x_1|)(|x_2|+|x_3|)}\Phi(x_2,x_3)f(x_1)+(-1)^{|f|(|x_1|+|x_3|)+|x_3|(|x_1|+|x_2|)}\Phi(x_3,x_1)f(x_2),
	\end{split}
\end{equation}
and for $f \in \mathbb{C}^1(\mathcal{G};\mathcal{V})$, $X_1=x_1\wedge x_2,~X_2=x_3\wedge x_4$, and $x_5\in \mathcal{G}$,
\begin{equation}\label{e26}
	\begin{split}
	&(\delta_{\Phi} f)(X_1,X_2,x_5)=-f([X_1,X_2]_\mathbb{F},x_5)-(-1)^{(|x_1|+|x_2|)(|x_3|+|x_4|)}f(X_2,[X_1,x_5])+f(X_1,[X_2,x_5])\\
	&+(-1)^{|f|(|x_1|+|x_2|)}\Phi(X_1)f(X_2,x_5)-(-1)^{(|f|+|x_1|+|x_2|)(|x_3|+|x_4|)}\Phi(X_2)f(X_1,x_5)\\
	&-(-1)^{(|f|+|x_1|+|x_2|+|x_3|)(|x_4|+|x_5|)}\Phi(x_4,x_5)f(X_1,x_3)\\
	&-(-1)^{(|f|+|x_1|+|x_2|)(|x_3|+|x_5|)+|x_5|(|x_3|+|x_4|)}\Phi(x_5,x_3)f(X_1,x_4)\\
	&=-f([x_1,x_2,x_3],x_4,x_5)-(-1)^{|x_3|(|x_1|+|x_2|)}f(x_3,[x_1,x_2,x_4],x_5)	\\
	&-(-1)^{(|x_1|+|x_2|)(|x_3|+|x_4|)}f(x_3,x_4,[x_1,x_2,x_5])+f(x_1,x_2,[x_3,x_4,x_5])\\
	&+(-1)^{|f|(|x_1|+|x_2|)}\Phi(x_1,x_2)f(x_3,x_4,x_5)-(-1)^{(|f|+|x_1|+|x_2|)(|x_3|+|x_4|)}\Phi(x_3,x_4)f(x_1,x_2,x_5)\\
	&-(-1)^{(|f|+|x_1|+|x_2|+|x_3|)(|x_4|+|x_5|)}\Phi(x_4,x_5)f(x_1,x_2,x_3)\\
	&-(-1)^{(|f|+|x_1|+|x_2|)(|x_3|+|x_5|)+|x_5|(|x_3|+|x_4|)}\Phi(x_5,x_3)f(x_1,x_2,x_4),
	\end{split}
\end{equation}
where $[.,.,.]_\mathbb{F}$ is given by Eq \ref{e22}.

\par Suppose that $\mathcal{L}$ and $\mathcal{P}$ are 3-Lie superalgebras. If $0\rightarrow \mathcal{P}\hookrightarrow \mathcal{L}\xrightarrow{\pi}\mathcal{Q}\rightarrow 0$ is an exact sequence of 3-Lie superalgebras and $[\mathcal{P},\mathcal{P},\mathcal{L}]=0$ then we call $\mathcal{L}$ an abelian extension of $\mathcal{L}$ by $\mathcal{P}$. An even linear map $s:\mathcal{Q}\rightarrow \mathcal{L}$ is called a section if  it satisfies $\pi s=1$. If there exists a section which is also a homomorphism between 3-Lie superalgebras, we say that the abelian extension splits. Now we construct a representation of $\mathcal{Q}$ on $\mathcal{P}$ and a cohomology class. Fix any section $s:\mathcal{Q}\rightarrow \mathcal{L}$ of $\pi$ and define $\Phi:\wedge^2 \mathcal{Q}\rightarrow End(\mathcal{P})$ by 
\begin{equation}\label{e27}
	\Phi(x,y)(v)=[s(x),s(y),v]_\mathcal{L},
\end{equation}
for $x,y \in \mathcal{Q}$ and $v\in \mathcal{P}$. It is easy to check that $\Phi$ is independent of the choice of $s$. Moreover, since
\begin{equation}\label{e28}
	[s(x),s(y),s(z)]_\mathcal{L}-s([x,y,z]_\mathcal{Q})\in \mathcal{P},
\end{equation}
for $x,y,z\in \mathcal{Q}$,
we have a map $\Omega:\wedge^3 \mathcal{Q}\rightarrow End(\mathcal{P})$ given by 
\begin{equation}\label{e29}
\Omega(x,y,z)=[s(x),s(y),s(z)]_\mathcal{L}-s([x,y,z]_\mathcal{Q})\in \mathcal{P},
\end{equation}
for $x,y,z\in \mathcal{Q}$.

\begin{lemma}\label{l1}
	Let $0\rightarrow \mathcal{P}\hookrightarrow \mathcal{L}\xrightarrow{\pi} \mathcal{Q}\rightarrow 0$ be an extension of $3$-Lie superalgebras with $[\mathcal{P},\mathcal{P},\mathcal{L}]=0$. Then
	\begin{enumerate}
		\item $\Phi$ given by Eq \ref{e27} is a representation of $\mathcal{Q}$ on $\mathcal{P}$.
		\item $\Omega$ given by Eq \ref{e29} is a 2-cocycle associated to $(\Phi;\mathcal{P})$.
	\end{enumerate}
\end{lemma}
\begin{proof}
	\begin{enumerate}
		\item By the equality
		\begin{equation}
			\begin{split}
				&[s(x_1),u,[s(y_1),s(y_2),s(y_3)]_\mathcal{L}]_\mathcal{L}=[[s(x_1),u,s(y_1)]_\mathcal{L},s(y_2),s(y_3)]_\mathcal{L}\\
				&~~~~~~+(-1)^{|y_1|(|x_1|+|u|)}[s(y_1),[s(x_1),u,s(y_2)]_\mathcal{L},s(y_3)]_\mathcal{L}\\&~~~~~~+(-1)^{(|x_1|+|u|)(|y_1|+|y_2|)}[s(y_1),s(y_2),[s(x_1),u,s(y_3)]_\mathcal{L}]_\mathcal{L},
			\end{split}
		\end{equation}
we have	\begin{equation*}
		\begin{split}
			&[s(x_1),u,[s(y_1),s(y_2),s(y_3)]_\mathcal{L}]_\mathcal{L}=[[s(x_1),u,s(y_1)]_\mathcal{L},s(y_2),s(y_3)]_\mathcal{L}\\
			&~~~~~~+(-1)^{|y_1|(|x_1|+|u|)}[s(y_1),[s(x_1),u,s(y_2)]_\mathcal{L},s(y_3)]_\mathcal{L}\\&~~~~~~+(-1)^{(|x_1|+|u|)(|y_1|+|y_2|)}[s(y_1),s(y_2),[s(x_1),u,s(y_3)]_\mathcal{L}]_\mathcal{L}\\
			&\implies \Phi(x_1,[y_1,y_2,y_3]_\mathcal{Q})=[\Phi(x_1,y_1)(u),s(y_2),s(y_3)]_\mathcal{L}\\&~~~~~~+(-1)^{|y_1|(|x_1|+|u|)}[s(y_1),\Phi(x_1,y_2)(u),s(y_3)\\&~~~~~~+(-1)^{(|y_1|+|y_2|)(|x_1|+|u|)}[s(y_1),s(y_2),\Phi(x_1,y_3)(u)]_\mathcal{L}\\
			&\implies \Phi(x_1,[y_1,y_2,y_3]_\mathcal{Q})=(-1)^{(|x_1|+|y_1|)(|y_2|+|y_3|)}\Phi(y_2,y_3)\Phi(x_1,y_1)\\&~~~~~-(-1)^{|x_1|(|y_1|+|y_3|)+|y_2| |y_3|}\Phi(y_1,y_3)\Phi(x_1,y_2)  +(-1)^{|x_1|(|y_1|+|y_2|))}\Phi(y_1,y_2)\Phi(x_1,y_3).
		\end{split}
	\end{equation*}
Therefore $\Phi$ is a representation of $\mathcal{Q}$ on $\mathcal{P}$. 
\item By the equality
\begin{equation}
	\begin{split}
		&[s(x_1),s(x_2),[s(y_1),s(y_2),s(y_3)]_\mathcal{L}]_\mathcal{L}=[[s(x_1),s(x_2),s(y_1)]_\mathcal{L},s(y_2),s(y_3)]_\mathcal{L}\\
		&~~~~~~+(-1)^{|y_1|(|x_1|+|x_2|)}[s(y_1),[s(x_1),s(x_2),s(y_2)]_\mathcal{L},s(y_3)]_\mathcal{L}\\&~~~~~~+(-1)^{(|x_1|+|x_2|)(|y_1|+|y_2|)}[s(y_1),s(y_2),[s(x_1),s(x_2),s(y_3)]_\mathcal{L}]_\mathcal{L},
	\end{split}
\end{equation}
\noindent we have
\begin{equation*}
	\begin{split}
			&[s(x_1),s(x_2),\Omega(y_1,y_2,y_3)]_\mathcal{L}+[s(x_1),s(x_2),s([y_1,y_2,y_3]_\mathcal{Q})]_\mathcal{L}\\
		&~~~~~~=[\Omega (x_1,x_2,y_1),s(y_2),s(y_3)]_\mathcal{L}+[s([x_1,x_2,y_1]_\mathcal{Q}),s(y_2),s(y_3)]_\mathcal{L}\\
		&~~~~~~+(-1)^{|y_1|(|x_1|+|x_2|)}[s(y_1),\Omega (x_1,x_2,y_2),s(y_3)]_\mathcal{L}\\&~~~~~~+(-1)^{|y_1|(|x_1|+|x_2|)}[s(y_1),s([x_1,x_2,y_2]_\mathcal{Q}),s(y_3)]\\
		&~~~~~~+(-1)^{(|x_1|+|x_2|)(|y_1|+|y_2|)}[s(y_1),s(y_2),\Omega (x_1,x_2,y_3)]_\mathcal{L}\\&~~~~~~+(-1)^{(|x_1|+|x_2|)(|y_1|+|y_2|)}[s(y_1),s(y_2),s([x_1,x_2,y_3]_\mathcal{Q})]_\mathcal{L}\\
		&\implies \Phi(x_1,x_2) (\Omega(y_1,y_2,y_3))+\Omega(x_1,x_2,[y_1,y_2,y_3]_\mathcal{Q})+s([x_1,x_2,[y_1,y_2,y_3]_\mathcal{Q}]_\mathcal{Q})\\
		&~~~~~~=\Phi(y_2,y_3) (\Omega(x_1,x_2,y_1))+\Omega([x_1,x_2,y_1]_\mathcal{Q},y_2,y_3)+s([x_1,x_2,y_1]_\mathcal{Q},y_2,y_3)\\
		&~~~~~~+(-1)^{|y_1|(|x_1|+|x_2|)}(\Phi(y_3,y_1) (\Omega(x_1,x_2,y_2))+(-1)^{|y_1|(|x_1|+|x_2|)}\Omega(y_1,[x_1,x_2,y_2]_\mathcal{Q},y_3)\\&~~~~~+(-1)^{ |y_1|(|x_1|+|x_2|)}s(y_1,[x_1,x_2,y_2]_\mathcal{Q},y_3))+(-1)^{ (|x_1|+|x_2|)(|y_1|+|y_2|)}(\Phi(y_1,y_2) (\Omega(x_1,x_2,y_3))\\&~~~~~~+(-1)^{ (|x_1|+|x_2|)(|y_1|+|y_2|)}\Omega(y_1,y_2,[x_1,x_2,y_3]_\mathcal{Q})+(-1)^{ (|x_1|+|x_2|)(|y_1|+|y_2|)}s[y_1,y_2,[x_1,x_2,y_3]_\mathcal{Q}]_\mathcal{Q})\\
		&\implies \Phi(x_1,x_2) (\Omega(y_1,y_2,y_3))+\Omega(x_1,x_2,[y_1,y_2,y_3]_\mathcal{Q})=\Phi(y_2,y_3) (\Omega(x_1,x_2,y_1))\\&~~~~~~+\Omega([x_1,x_2,y_1]_\mathcal{Q},y_2,y_3)+(-1)^{|y_1|(|x_1|+|x_2|)}(\Phi(y_3,y_1) (\Omega(x_1,x_2,y_2))\\&~~~~~~~+(-1)^{ |y_1|(|x_1|+|x_2|)}\Omega(y_1,[x_1,x_2,y_2]_\mathcal{Q},y_3)+(-1)^{ (|x_1|+|x_2|)(|y_1|+|y_2|)}(\Phi(y_1,y_2) (\Omega(x_1,x_2,y_3))\\&~~~~~~~+(-1)^{ (|x_1|+|x_2|)(|y_1|+|y_2|)}\Omega(y_1,y_2,[x_1,x_2,y_3]_\mathcal{Q}).
	\end{split}
\end{equation*}
\noindent Hence $\Omega$ is a 2-cocycle associated to $(\Phi;\mathcal{P})$.
	\end{enumerate}
\end{proof}

\begin{corollary}\label{coro21}
Let	$0\rightarrow \mathcal{P}\hookrightarrow \mathcal{L}\xrightarrow{\pi} \mathcal{Q}\rightarrow 0$ be an extension of $3$-Lie superalgebras with $[\mathcal{P},\mathcal{P},\mathcal{L}]=0$. Then the cohomology class $[\Omega]$ does not depend on the choice of $s$.
\end{corollary}
\begin{proof}
	Let $s_1$ and $s_2$ be sections of $\pi$ and $\Omega_1,\Omega_2$ be defined by Eq \ref{e29} which are corresponding to $s_1,s_2$, respectively. For any $x\in \mathcal{Q}$, set $\lambda(x)=s_1(x)-s_2(x)$.\\
	$(\pi \lambda)(x)=x-x=0$, $\lambda(x)\in \mathcal{P}$, and  $\lambda\in (\mathbb{C}^0(\mathcal{G};\mathcal{V}))_{\overline{0}}$. Then 
	\begin{equation*}
		\begin{split}
			&\Omega_1(x,y,z)-\Omega_2(x,y,z)\\
			&=[s_1(x),s_1(y),s_1(z)]_\mathcal{L}-s_1([x,y,z]_\mathcal{Q})-[s_2(x),s_2(y),s_2(z)]_\mathcal{L}+s_2([x,y,z]_\mathcal{Q})\\
			&=[s_2(x)+\lambda(x),s_2(y)+\lambda(y),s_2(z)+\lambda(z)]_\mathcal{L}-s_2([x,y,z]_\mathcal{Q})+\lambda([x,y,z]_\mathcal{Q})\\
			&~~~~~~-[s_2(x),s_2(y),s_2(z)]_\mathcal{L}+s_2([x,y,z]_\mathcal{Q})\\
			&=[s_2(x),s_2(y),\lambda(z)]_\mathcal{L}+[\lambda(x),s_2(y),s_2(z)]_\mathcal{L}+[s_2(x),\lambda(y),s_2(z)]_\mathcal{L}-\lambda([x,y,z]_\mathcal{Q})	\\
			&=-\lambda([x,y,z]_\mathcal{Q})+\Phi(x,y)(\lambda(z))+(-1)^{|x|(|y|+|z|)}\Phi(y,z)(\lambda(x))\\
			&~~~~~~~+(-1)^{|z|(|x|+|y|)}\Phi(z,x)(\lambda(y))\\
			&=(\delta_{\Phi}\lambda)(x,y,z),
		\end{split}
	\end{equation*}
which completes the proof.
\end{proof}

\begin{proposition}
	 If $(\Phi, \mathcal{P})$ is a representation of $\mathcal{Q}$ and $\Omega$ is a $2$-cocycle given by the representation $(\Phi, \mathcal{P})$, then $\mathcal{L}_{\Phi,\Omega}:=\mathcal{Q}\oplus \mathcal{P}$ is a 3-Lie superalgebra with the superbracket given by 
	\begin{equation}\label{e210}
		\begin{split}
			[x+u,y+v,z+w]_{\mathcal{L}_{\Phi,\Omega}}=&[x,y,z]_{\mathcal{Q}}+\Omega(x,y,z)+\Phi(x,y)(w)\\&~~~~~~+(-1)^{|y||z|}\Phi(z,x)(v)+(-1)^{|x|(|y|+|z|)}\Phi(y,z)(u),
		\end{split}
	\end{equation}
where $x,y,z\in \mathcal{Q}$ and $u,v,w\in \mathcal{P}$.
\end{proposition}
\begin{proof}
	It is sufficient to verify the fundamental identity. So, we have
	\begin{equation}\label{e211}
		\begin{split}
			&[x_1+u_1,x_2+u_2,[y_1+v_1,y_2+v_2,y_3+v_3]_{\mathcal{L}_{\Phi,\Omega}}]_{\mathcal{L}_{\Phi,\Omega}}\\
			&=[x_1,x_2,[y_1,y_2,y_3]_\mathcal{Q}]_\mathcal{Q}+\Omega(x_1,x_2,[y_1,y_2,y_3]_\mathcal{Q})+\Phi(x_1,x_2)(\Omega(y_1,y_2,y_3)\\
			&~~~~~~+\Phi(y_1,y_2)(v_3)+(-1)^{|y_2||y_3|}\Phi(y_3,y_1)(v_2)+(-1)^{|y_1|(|y_2|+|y_3|)}\Phi(y_2,y_3)(v_1))\\&~~~~~~~+(-1)^{|x_2|(|y_1|+|y_2|+|y_3|)}\Phi([y_1,y_2,y_3]_\mathcal{Q},x_1)(u_2)\\&~~~~~~+(-1)^{|x_1|(|x_2|+|y_1|+|y_2|+|y_3|)}\Phi(x_2,[y_1,y_2,y_3]_\mathcal{Q})(u_1),
		\end{split}
	\end{equation}

	\begin{equation}\label{e212}
	\begin{split}
		&[[x_1+u_1,x_2+u_2,y_1+v_1]_{\mathcal{L}_{\Phi,\Omega}},y_2+v_2,y_3+v_3]_{\mathcal{L}_{\Phi,\Omega}}\\
		&=[[x_1,x_2,y_1]_\mathcal{Q},y_2,y_3]_\mathcal{Q}+\Omega([x_1,x_2,y_1]_\mathcal{Q},y_2,y_3)\\&~~~~~~+(-1)^{(|x_1|+|x_2|+|y_1|)(|y_2|+|y_3|)}\Phi(y_2,y_3)(\Omega(x_1,x_2,y_1)\\
		&~~~~~~+\Phi(x_1,x_2)(v_1)+(-1)^{|x_2||y_1|}\Phi(y_1,x_1)(u_2)+(-1)^{|x_1|(|x_2|+|y_1|)}\Phi(x_2,y_1)(u_1))\\&~~~~~~~+\Phi([x_1,x_2,y_1]_\mathcal{Q},y_2)(v_3)+(-1)^{|y_3|(|x_1|+|x_2|+|y_2|)}\Phi(y_3,[x_1,x_2,y_1]_\mathcal{Q})(v_2),
	\end{split}
\end{equation}

\begin{equation}\label{e213}
	\begin{split}
			&(-1)^{|y_1|(|x_1|+|x_2|)}[y_1+v_1,[x_1+u_1,x_2+u_2,y_2+v_2]_{\mathcal{L}_{\Phi,\Omega}},y_3+v_3]_{\mathcal{L}_{\Phi,\Omega}}\\
		&=(-1)^{|y_1|(|x_1|+|x_2|)}[y_1,[x_1,x_2,y_2]_\mathcal{Q},y_3]_\mathcal{Q}+(-1)^{|y_1|(|x_1|+|x_2|)}\Omega(y_1,[x_1,x_2,y_2]_\mathcal{Q},y_3)	\\&~~~~~~+(-1)^{(|x_1|+|x_2|+|y_1|)(|y_2|+|y_3|)+|y_1|(|x_1|+|x_2|)}\Phi(y_3,y_1)(\Omega(x_1,x_2,y_2)\\
		&~~~~~~+\Phi(x_1,x_2)(v_2)+(-1)^{|x_2||y_2|}\Phi(y_2,x_1)(u_2)+(-1)^{|x_1|(|x_2|+|y_2|)}\Phi(x_2,y_2)(u_1))\\&~~~~~~~+(-1)^{|y_1|(|x_1|+|x_2|)}\Phi(y_1,[x_1,x_2,y_2]_\mathcal{Q})(v_3)\\	&~~~~~~+(-1)^{|y_3|(|x_1|+|x_2|+|y_2|)+|y_1|(|x_1|+|x_2|)}\Phi([x_1,x_2,y_2]_\mathcal{Q},y_3)(v_1),
	\end{split}\end{equation}
and
\begin{equation}\label{e214}
	\begin{split}
		&(-1)^{(|x_1|+|x_2|)(|y_1|+|y_2|)}[y_1+v_1,y_2+v_2,[x_1+u_1,x_2+u_2,y_3+v_3]_{\mathcal{L}_{\Phi,\Omega}}]_{\mathcal{L}_{\Phi,\Omega}}\\
		&=(-1)^{(|x_1|+|x_2|)(|y_1|+|y_2|)}[y_1,y_2,[x_1,x_2,y_3]_\mathcal{Q}]_\mathcal{Q}\\&~~~~~~ +(-1)^{(|x_1|+|x_2|)(|y_1|+|y_2|)}\Omega(y_1,y_2,[x_1,x_2,y_3]_\mathcal{Q})	\\&~~~~~~+(-1)^{(|x_1|+|x_2|)(|y_1|+|y_2|)}\Phi(y_1,y_2)(\Omega(x_1,x_2,y_3)+\Phi(x_1,x_2)(v_3)\\
		&~~~~~~+(-1)^{|x_2||y_3|}\Phi(y_3,x_1)(u_2)+(-1)^{|x_1|(|x_2|+|y_3|)}\Phi(x_2,y_3)(u_1))\\&~~~~~~~+(-1)^{|y_2|(|x_1|+|x_2|+|y_3|)+(|x_1|+|x_2|)(|y_1|+|y_2|)}\Phi([x_1,x_2,y_3]_\mathcal{Q},y_1)(v_2)\\	&~~~~~~+(-1)^{|y_1|(|x_1|+|x_2|+|y_2|+|y_3|)+(|x_1|+|x_2|)(|y_1|+|y_2|)}\Phi(y_2,[x_1,x_2,y_3]_\mathcal{Q})(v_1).
\end{split}\end{equation}
We will see that Eq $(\ref{e211})$=Eq $(\ref{e212})$+Eq $(\ref{e213})$+Eq $(\ref{e214})$, if $\Phi$ is a representation and $\Omega$ is a $2$-cocycle.
Hence, we have
\begin{equation}\label{e215}
	\begin{split}
		&[x_1+u_1,x_2+u_2,[y_1+v_1,y_2+v_2,y_3+v_3]_{\mathcal{L}_{\Phi,\Omega}}]_{\mathcal{L}_{\Phi,\Omega}}\\&=[[x_1+u_1,x_2+u_2,y_1+v_1]_{\mathcal{L}_{\Phi,\Omega}},y_2+v_2,y_3+v_3]_{\mathcal{L}_{\Phi,\Omega}}\\&~~~~~~+(-1)^{|y_1|(|x_1|+|x_2|)}[y_1+v_1,[x_1+u_1,x_2+u_2,y_2+v_2]_{\mathcal{L}_{\Phi,\Omega}},y_3+v_3]_{\mathcal{L}_{\Phi,\Omega}}\\&~~~~~~~+
		(-1)^{(|x_1|+|x_2|)(|y_1|+|y_2|)}[y_1+v_1,y_2+v_2,[x_1+u_1,x_2+u_2,y_3+v_3]_{\mathcal{L}_{\Phi,\Omega}}]_{\mathcal{L}_{\Phi,\Omega}},
	\end{split}
\end{equation}
which implies $\mathcal{L}_{\Phi,\Omega}$ is a 3-Lie superalgebra.
\end{proof}

\begin{proposition}\label{p2}
	If $\mathbb{H}^1(\mathcal{Q};\mathcal{P})=0$ then the extension splits.
\end{proposition}
\begin{proof}
	It is sufficient to show that there is a section of $\pi$ which is a 3-Lie superalgebra homomorphism. It is known that $(\Phi;\mathcal{P})$ given by Eq \ref{e27}  is independent of the choice of the sections of $\pi$. Consider the 2-cocycle $\Omega$ given by Lemma \ref{l1}. Since $\mathbb{H}^1(\mathcal{Q};\mathcal{P})=0$, there exists an $\xi\in (\mathbb{C}^0(\mathcal{Q};\mathcal{P}))_{\overline{0}}$ such that $\Omega=\delta_{\Phi}\xi$. For any $x,y,z\in \mathcal{Q}$, it follows that
	\begin{equation}
	\begin{split}
			\Omega(x,y,z)=	&-\xi([x,y,z]_\mathcal{Q})+\Phi(x,y)(\xi(z))+(-1)^{|x|(|y|+|z|)}\Phi(y,z)(\xi(x))\\
			&+(-1)^{|z|(|x|+|y|)}\Phi(z,x)(\xi(y)).
	\end{split}
	\end{equation}
Define an even linear map $s^{'}:\mathcal{Q}\rightarrow \mathcal{P}$ by 
$$s^{'}(x)=s(x)-\xi(x).$$
Note that $s^{'}$ is also a section of $\pi$. Then for any $x,y,z\in \mathcal{Q}$, we have
	\begin{equation*}
		\begin{split}
			&~~~~~~~~~[s^{'}(x),s^{'}(y),s^{'}(z)]_\mathcal{L}\\
			&=[s(x)-\xi(x),s(y)-\xi(y),s(z)-\xi(z)]_\mathcal{L}\\
			&=[s_1(x),s_1(y),s_1(z)]_\mathcal{L}-\Phi(x,y)(\xi(z))-(-1)^{|x|(|y|+|z|)}\Phi(y,z)(\xi(x))-(-1)^{|z|(|x|+|y|)}\Phi(z,x)(\xi(y))\\
			&=s([x,y,z]_\mathcal{Q})+\Omega(x,y,z)-\Phi(x,y)(\xi(z))-(-1)^{|x|(|y|+|z|)}\Phi(y,z)(\xi(x))-(-1)^{|z|(|x|+|y|)}\Phi(z,x)(\xi(y))\\
			&=s([x,y,z]_\mathcal{Q})-\xi([x,y,z]_\mathcal{Q})\\
			&=s^{'}([x,y,z]_\mathcal{Q}).
		\end{split}
	\end{equation*}
 Hence $s^{'}$ is a $3$-Lie superalgebra homomorphism. 
\end{proof}

\section{Cohomology Classes and a Lie superalgebra}
Let $(\Phi;\mathcal{P})$ be a representation of a $3$-Lie superalgebra $\mathcal{Q}$ on $\mathcal{P}$. In this section, we use superderivations of $\mathcal{P}$, $\mathcal{Q}$ to construct first-order cohomology classes. By using this, we construct a Lie superalgebra and its representation on $\mathbb{H}^1(\mathcal{Q};\mathcal{P})$.
\par Let $\mathcal{P},\mathcal{Q}$ be $3$-Lie superalgebras. Given a representation $(\Phi;\mathcal{P})$ of $\mathcal{Q}$. Suppose $\Omega \in (\mathbb{C}^1(\mathcal{Q};\mathcal{P}))_{\overline{0}}$. For any pair $(\mathcal{D}_p,\mathcal{D}_q)\in Der(\mathcal{P})\times Der(\mathcal{Q})$, define a $2$-cochain $Ob^\Omega _{(\mathcal{D}_p,\mathcal{D}_q)} \in \mathbb{C}^1(\mathcal{Q};\mathcal{P})$ as 
\begin{equation}\label{e31}
	Ob^\Omega _{(\mathcal{D}_p,\mathcal{D}_q)}=\mathcal{D}_p\Omega-\Omega(\mathcal{D}_q\otimes 1\otimes 1)-\Omega(1\otimes \mathcal{D}_q\otimes  1)-\Omega(1\otimes 1\otimes  \mathcal{D}_q),
\end{equation}
where ``1'' denotes the identity map and the degree of identity map is always even. This is equivalent to 
\begin{equation}\label{e32}
	\begin{split}
		Ob^\Omega _{(\mathcal{D}_p,\mathcal{D}_q)}(x,y,z)=&\mathcal{D}_p(\Omega(x,y,z))-\Omega(\mathcal{D}_q(x),y,z)-(-1)^{|\alpha| |x|}\Omega(x,\mathcal{D}_q(y),z)\\&-(-1)^{|\alpha| (|x|+|y|)}\Omega(x,y,\mathcal{D}_q(z)),
	\end{split}
\end{equation}
for $x,y,z\in \mathcal{Q}$ and $|Ob^\Omega _{(\mathcal{D}_p,\mathcal{D}_q)}|=|\mathcal{D}_p|=|\mathcal{D}_q|=\alpha$ where $\alpha \in \mathbb{Z}_2$.
\par We begin with the following.
\begin{lemma}\label{l31}
	Let $(\Phi;\mathcal{P})$ be a representation of $\mathcal{Q}$ and $\Omega \in (\mathbb{C}^1(\mathcal{G};\mathcal{V}))_{\overline{0}}$ associated to the representation $(\Phi;\mathcal{P})$. Assume that a pair $(\mathcal{D}_p,\mathcal{D}_q)\in Der(\mathcal{P})\times Der(\mathcal{Q})$ satisfies that 
	\begin{equation}\label{e33}
		\mathcal{D}_p\Phi(x,y)-(-1)^{|\alpha|(|x|+|y|)}\Phi(x,y)\mathcal{D}_p=\Phi(\mathcal{D}_q(x),y)+(-1)^{|\alpha||x|}\Phi(x,\mathcal{D}_q(y)),
	\end{equation} 
for $x,y\in \mathcal{Q}$ and $|\mathcal{D}_p|=|\mathcal{D}_q|=\alpha$. If $\Omega$ is a $2$-cocycle then $Ob^\Omega _{(\mathcal{D}_p,\mathcal{D}_q)}\in \mathbb{C}^1(\mathcal{G};\mathcal{V})$ given by Eq \ref{e32} is also a $2$-cocycle.
\end{lemma}
\begin{proof}
	It is sufficent to show that $\delta_\Phi Ob^\Omega _{(\mathcal{D}_p,\mathcal{D}_q)}=0$. Since $\Omega$ is a $2$-cocycle, so $\delta_\Phi \Omega=0$, by Eq \ref{e26} it follows that, for any $x_i\in \mathcal{Q}$
	
\begin{equation} \label{e34}
\begin{split}
		0=&-\Omega ([x_1,x_2,x_3]_\mathcal{Q},x_4,x_5)-(-1)^{|x_3|(|x_1|+|x_2|)}\Omega(x_3,[x_1,x_2,x_4]_\mathcal{Q},x_5)\\&-(-1)^{(|x_3|+|x_4|)(|x_1|+|x_2|)}\Omega(x_3,x_4,[x_1,x_2,x_5]_\mathcal{Q})+\Omega(x_1,x_2,[x_3,x_4,x_5]_\mathcal{Q})\\
		&+\Phi(x_1,x_2)\Omega(x_3,x_4,x_5)-(-1)^{(|x_3|+|x_4|)(|x_1|+|x_2|)}\Phi(x_3,x_4)\Omega(x_1,x_2,x_5)\\
		&-(-1)^{(|x_4|+|x_5|)(|x_1|+|x_2|+|x_3|)}\Phi(x_4,x_5)\Omega(x_1,x_2,x_3)\\&-(-1)^{(|x_3|+|x_5|)(|x_1|+|x_2|)+(|x_3|+|x_4|)|x_5|}\Phi(x_5,x_3)\Omega(x_1,x_2,x_4).
\end{split}
\end{equation}	
Now $|Ob^\Omega _{(\mathcal{D}_p,\mathcal{D}_q)}|=|\mathcal{D}_p|=|\mathcal{D}_q|=\alpha$, we have	
	\begin{equation}\label{e35}
	\begin{split}
			\delta_\Phi &Ob^\Omega _{(\mathcal{D}_p,\mathcal{D}_q)}(x_1,x_2,x_3,x_4,x_5)\\&=-\underbrace{Ob^\Omega _{(\mathcal{D}_p,\mathcal{D}_q)}([x_1,x_2,x_3]_\mathcal{Q},x_4,x_5)}_{(1)}-\underbrace{(-1)^{|x_3|(|x_1|+|x_2|)}Ob^\Omega _{(\mathcal{D}_p,\mathcal{D}_q)}(x_3,[x_1,x_2,x_4]_\mathcal{Q},x_5)}_{(2)}\\&
			-\underbrace{(-1)^{(|x_3|+|x_4|)(|x_1|+|x_2|)}Ob^\Omega _{(\mathcal{D}_p,\mathcal{D}_q)}(x_3,x_4,[x_1,x_2,x_5]_\mathcal{Q})}_{(3)}
			+\underbrace{Ob^\Omega _{(\mathcal{D}_p,\mathcal{D}_q)}(x_1,x_2,[x_3,x_4,x_5]_\mathcal{Q})}_{(4)}\\
			&+\underbrace{(-1)^{\alpha(|x_1|+|x_2|)}\Phi(x_1,x_2)Ob^\Omega _{(\mathcal{D}_p,\mathcal{D}_q)}(x_3,x_4,x_5)}_{(5)}\\&-\underbrace{(-1)^{(|x_3|+|x_4|)(\alpha+|x_1|+|x_2|)}\Phi(x_3,x_4)Ob^\Omega _{(\mathcal{D}_p,\mathcal{D}_q)}(x_1,x_2,x_5)}_{(6)}\\
			&-\underbrace{(-1)^{(|x_4|+|x_5|)(\alpha+|x_1|+|x_2|+|x_3|)}\Phi(x_4,x_5)Ob^\Omega _{(\mathcal{D}_p,\mathcal{D}_q)}(x_1,x_2,x_3)}_{(7)}\\&-\underbrace{(-1)^{(|x_3|+|x_5|)(\alpha+|x_1|+|x_2|)+(|x_3|+|x_4|)|x_5|}\Phi(x_5,x_3)Ob^\Omega _{(\mathcal{D}_p,\mathcal{D}_q)}(x_1,x_2,x_4)}_{(8)}.
		\end{split}
		\end{equation}	
Applying the definition of $Ob^\Omega _{(\mathcal{D}_p,\mathcal{D}_q)}$ as in Eq \ref{e32}, we get 	
\begin{equation}\label{e36}
		\begin{split}
			(1)&=-\mathcal{D}_p(\Omega([x_1,x_2,x_3]_\mathcal{Q},x_4,x_5))+\Omega([\mathcal{D}_q(x_1),x_2,x_3]_\mathcal{Q},x_4,x_5)\\
			&~~~~~~+(-1)^{\alpha|x_1|}\Omega([x_1,\mathcal{D}_q(x_2),x_3]_\mathcal{Q},x_4,x_5)\\&~~~~~~+(-1)^{\alpha(|x_1|+|x_2|)}\Omega([x_1,x_2,\mathcal{D}_q(x_3)]_\mathcal{Q},x_4,x_5)\\
			&~~~~~~+(-1)^{\alpha(|x_1|+|x_2|+|x_3|)}\Omega([x_1,x_2,x_3]_\mathcal{Q},\mathcal{D}_q(x_4),x_5)\\
			&~~~~~~+(-1)^{\alpha(|x_1|+|x_2|+|x_3|+|x_4|)}\Omega([x_1,x_2,x_3]_\mathcal{Q},x_4,\mathcal{D}_q(x_5)),
		\end{split}
	\end{equation}
\begin{equation}\label{e37}
	\begin{split}
		(2)
		&=-(-1)^{|x_3|(|x_1|+|x_2|)}\mathcal{D}_p(\Omega(x_3,[x_1,x_2,x_4]_\mathcal{Q},x_5))\\&~~~~~~+(-1)^{|x_3|(|x_1|+|x_2|)}\Omega(\mathcal{D}_q(x_3),[x_1,x_2,x_4]_\mathcal{Q},x_5)\\
		&~~~~~~+(-1)^{|x_3|(|x_1|+|x_2|)+\alpha|x_3|}\Omega(x_3,[\mathcal{D}_q(x_1),x_2,x_4]_\mathcal{Q},x_5)\\&~~~~~~+(-1)^{|x_3|(|x_1|+|x_2|)+\alpha(|x_1|+|x_3|)}\Omega(x_3,[x_1,\mathcal{D}_q(x_2),x_4]_\mathcal{Q},x_5)\\&~~~~~~+(-1)^{|x_3|(|x_1|+|x_2|)+\alpha(|x_1|+|x_2|+|x_3|)}\Omega(x_3,[x_1,x_2,\mathcal{D}_q(x_4)]_\mathcal{Q},x_5)\\
		&~~~~~~+(-1)^{|x_3|(|x_1|+|x_2|)+\alpha(|x_1|+|x_2|+|x_3|+|x_4|)}\Omega(x_3,[x_1,x_2,x_4]_\mathcal{Q},\mathcal{D}_q(x_5)),
	\end{split}
	\end{equation}
	
\begin{equation}\label{e38}
	\begin{split}
		(3)
		&=-(-1)^{(|x_1|+|x_2|)(|x_3|+|x_4|)}\mathcal{D}_p(\Omega (x_3,x_4,[x_1,x_2,x_5]_\mathcal{Q}))\\&~~~~~~+(-1)^{(|x_1|+|x_2|)(|x_3|+|x_4|)}\Omega (\mathcal{D}_q (x_3),x_4,[x_1,x_2,x_5]_\mathcal{Q})\\
		&~~~~~~+(-1)^{(|x_1|+|x_2|)(|x_3|+|x_4|)+\alpha|x_3|}\Omega (x_3,\mathcal{D}_q (x_4),[x_1,x_2,x_5]_\mathcal{Q})\\
		&~~~~~~+(-1)^{(|x_1|+|x_2|)(|x_3|+|x_4|)+\alpha(|x_3|+|x_4|)}\Omega (x_3,x_4,[\mathcal{D}_q (x_1),x_2,x_5]_\mathcal{Q})\\
		&~~~~~~+(-1)^{(|x_1|+|x_2|)(|x_3|+|x_4|)+\alpha(|x_1|+|x_3|+|x_4|)}\Omega (x_3,x_4,[x_1,\mathcal{D}_q (x_2),x_5]_\mathcal{Q})\\
		&~~~~~~+(-1)^{(|x_1|+|x_2|)(|x_3|+|x_4|)+\alpha(|x_1|+|x_2|+|x_3|+|x_4|)}\Omega (x_3,x_4,[x_1,x_2,\mathcal{D}_q (x_5)]_\mathcal{Q}),
	\end{split}
\end{equation}	
	
\begin{equation}\label{e39}
	\begin{split}
	(4)&=\mathcal{D}_p(\Omega(x_1,x_2,[x_3,x_4,x_5]_\mathcal{Q}))-\Omega (\mathcal{D}_q(x_1),x_2,[x_3,x_4,x_5]_\mathcal{Q})\\
		&~~~~~~-(-1)^{\alpha|x_1|}\Omega (x_1,\mathcal{D}_q(x_2),[x_3,x_4,x_5]_\mathcal{Q})\\&~~~~~~-(-1)^{\alpha(|x_1|+|x_2|)}\Omega (x_1,x_2,[\mathcal{D}_q(x_3),x_4,x_5]_\mathcal{Q})\\&~~~~~~-(-1)^{\alpha(|x_1|+|x_2|+|x_3|)}\Omega (x_1,x_2,[x_3,\mathcal{D}_q(x_4),x_5]_\mathcal{Q})\\&~~~~~~
		-(-1)^{\alpha(|x_1|+|x_2|+|x_3|+|x_4|)}\Omega (x_1,x_2,[x_3,x_4,\mathcal{D}_q(x_5)]_\mathcal{Q}),
	\end{split}
\end{equation}

\begin{equation}\label{e310}
	\begin{split}
		(5)&=-(-1)^{\alpha(|x_1|+|x_2|)}\Phi(x_1,x_2)(\mathcal{D}_p(\Omega(x_3,x_4,x_5))-\Omega(\mathcal{D}_q(x_3),x_4,x_5)\\
		&~~~~~~-(-1)^{\alpha|x_3|}\Omega(x_3,\mathcal{D}_q(x_4),x_5)-(-1)^{\alpha(|x_3|+|x_4|)}\Omega(x_3,x_4,\mathcal{D}_q(x_5))),
	\end{split}
\end{equation}	
	\begin{equation}\label{e311}
		\begin{split}
			(6)&=-(-1)^{(|x_3|+|x_4|)(\alpha+|x_1|+|x_2|)}\Phi(x_3,x_4)(\mathcal{D}_p(\Omega(x_1,x_2,x_5))-\Omega(\mathcal{D}_q(x_1),x_2,x_5)\\
			&~~~~~~-(-1)^{\alpha|x_1|}\Omega(x_1,\mathcal{D}_q(x_2),x_5)-(-1)^{\alpha(|x_1|+|x_2|)}\Omega(x_1,x_2,\mathcal{D}_q(x_5))),
		\end{split}
	\end{equation}	
	\begin{equation}\label{e312}
	\begin{split}
		(7)&=-(-1)^{(|x_4|+|x_5|)(\alpha+|x_1|+|x_2|+|x_3|)}\Phi(x_4,x_5)(\mathcal{D}_p(\Omega(x_1,x_2,x_3))-\Omega(\mathcal{D}_q(x_1),x_2,x_3)\\
		&~~~~~~-(-1)^{\alpha|x_1|}\Omega(x_1,\mathcal{D}_q(x_2),x_3)-(-1)^{\alpha(|x_1|+|x_2|)}\Omega(x_1,x_2,\mathcal{D}_q(x_3))),
	\end{split}
\end{equation}	
\begin{equation}\label{e313}
	\begin{split}
		(8)&=-(-1)^{(|x_3|+|x_5|)(\alpha+|x_1|+|x_2|)+|x_5|(|x_3|+|x_4|)}\Phi(x_5,x_3)(\mathcal{D}_p(\Omega(x_1,x_2,x_4))-\Omega(\mathcal{D}_q(x_1),x_2,x_4)\\
		&~~~~~~-(-1)^{\alpha|x_1|}\Omega(x_1,\mathcal{D}_q(x_2),x_4)-(-1)^{\alpha(|x_1|+|x_2|)}\Omega(x_1,x_2,\mathcal{D}_q(x_4))).
	\end{split}
\end{equation}	
	
	\begin{equation}\label{e314}
		\begin{split}
		&-\mathcal{D}_p(\Omega([x_1,x_2,x_3]_\mathcal{Q},x_4,x_5))-(-1)^{|x_3|(|x_1|+|x_2|)}\mathcal{D}_p(\Omega(x_3,[x_1,x_2,x_4]_\mathcal{Q},x_5))\\&-(-1)^{(|x_1|+|x_2|)(|x_3|+|x_4|)}\mathcal{D}_p(\Omega(x_3,x_4,[x_1,x_2,x_5]_\mathcal{Q}))+\mathcal{D}_p(\Omega (x_1,x_2,[x_3,x_4,x_5]_\mathcal{Q}))\\
			&=-\mathcal{D}_p(\Phi(x_1,x_2)(\Omega(x_3,x_4,x_5)))+(-1)^{(|x_1|+|x_2|)(|x_3|+|x_4|)}\mathcal{D}_p(\Phi(x_3,x_4)(\Omega(x_1,x_2,x_5)))\\&+(-1)^{(|x_1|+|x_2|+|x_3|)(|x_4|+|x_5|)}\mathcal{D}_p(\Phi(x_4,x_5)(\Omega(x_1,x_2,x_3)))\\& +(-1)^{(|x_1|+|x_2|)(|x_3|+|x_5|)+|x_5|(|x_3|+|x_4|)}\mathcal{D}_p(\Phi(x_5,x_3)(\Omega(x_1,x_2,x_4))).
		\end{split}
	\end{equation}
	For Eq \ref{e36}-\ref{e314}, by suitable combinations and with the aid of Eq \ref{e34}, we get
\begin{equation}\label{e315}
	\begin{split}
		&\Omega([\mathcal{D}_q(x_1),x_2,x_3]_\mathcal{Q},x_4,x_5)+(-1)^{|x_3|(|x_1|+|x_2|)+\alpha|x_3|}\Omega(x_3,[\mathcal{D}_q(x_1),x_2,x_4]_\mathcal{Q},x_5)\\
		&+(-1)^{(|x_1|+|x_2|)(|x_3|+|x_4|)+\alpha(|x_3|+|x_4|)}\Omega (x_3,x_4,[\mathcal{D}_q (x_1),x_2,x_5]_\mathcal{Q})\\
		&-\Omega (\mathcal{D}_q(x_1),x_2,[x_3,x_4,x_5]_\mathcal{Q})+(-1)^{(|x_3|+|x_4|)(\alpha+|x_1|+|x_2|)}\Phi(x_3,x_4)\Omega(\mathcal{D}_q(x_1),x_2,x_5)\\
		&+(-1)^{(|x_4|+|x_5|)(\alpha+|x_1|+|x_2|+|x_3|)}\Phi(x_4,x_5)\Omega(\mathcal{D}_q(x_1),x_2,x_3)\\&-(-1)^{(|x_3|+|x_5|)(\alpha+|x_1|+|x_2|)+|x_5|(|x_3|+|x_4|)}\Phi(x_5,x_3)\Omega(\mathcal{D}_q(x_1),x_2,x_4)\\
		&=\Phi(\mathcal{D}_qx_1,x_2)\Omega(x_3,x_4,x_5),
	\end{split}
\end{equation}

\begin{equation}\label{e316}
	\begin{split}
		&(-1)^{\alpha|x_1|}\Omega([x_1,\mathcal{D}_q(x_2),x_3]_\mathcal{Q},x_4,x_5)+(-1)^{(|x_3|+\alpha)(|x_1|+|x_2|)}\Omega(x_3,[x_1,\mathcal{D}_q(x_2),x_4]_\mathcal{Q},x_5)\\
		&+(-1)^{(|x_1|+|x_2|+\alpha)(|x_3|+|x_4|)+\alpha|x_1|}\Omega (x_3,x_4,[x_1,\mathcal{D}_q (x_2),x_5]_\mathcal{Q})\\
		&-(-1)^{\alpha|x_1|}\Omega(x_1,\mathcal{D}_q(x_2),[x_3,x_4,x_5]_\mathcal{Q})\\
		&+(-1)^{(|x_3|+|x_4|)(\alpha+|x_1|+|x_2|)+\alpha|x_1|}\Phi(x_3,x_4)\Omega(x_1,\mathcal{D}_q(x_2),x_5)\\
		&+(-1)^{(|x_4|+|x_5|)(\alpha+|x_1|+|x_2|+|x_3|)+\alpha|x_1|}\Phi(x_4,x_5)\Omega(x_1,\mathcal{D}_q(x_2),x_3)\\
		&+(-1)^{(|x_3|+|x_5|)(\alpha+|x_1|+|x_2|)+|x_5|(|x_3|+|x_4|)+\alpha|x_1|}\Phi(x_5,x_3)\Omega(x_1,\mathcal{D}_q(x_2),x_4)\\
		&=(-1)^{\alpha|x_1|}\Phi(x_1,\mathcal{D}_q(x_2))\Omega(x_3,x_4,x_5),
	\end{split}
\end{equation}

\begin{equation}\label{e317}
	\begin{split}
		&(-1)^{\alpha(|x_1|+|x_2|)}\Omega([x_1,x_2,\mathcal{D}_q(x_3)]_\mathcal{Q},x_4,x_5)+(-1)^{|x_3|(|x_1|+|x_2|)}\Omega(\mathcal{D}_q(x_3),[x_1,x_2,x_4]_\mathcal{Q},x_5)\\
		&+(-1)^{(|x_1|+|x_2|)(|x_3|+|x_4|)}\Omega (\mathcal{D}_q (x_3),x_4,[x_1,x_2,x_5]_\mathcal{Q})-(-1)^{\alpha(|x_1|+|x_2|)}\Omega (x_1,x_2,[\mathcal{D}_q(x_3),x_4,x_5]_\mathcal{Q})\\&+(-1)^{\alpha(|x_1|+|x_2|)}\Phi(x_1,x_2)\Omega(\mathcal{D}_q(x_3),x_4,x_5)\\&+(-1)^{(|x_4|+|x_5|)(\alpha+|x_1|+|x_2|+|x_3|)+\alpha(|x_1|+|x_2|)}\Phi(x_4,x_5)\Omega(x_1,x_2,\mathcal{D}_q(x_3))\\
		&=-(-1)^{(|x_1|+|x_2|)(|x_3|+|x_4|)}\Phi(\mathcal{D}_q(x_3),x_4)\Omega(x_1,x_2,x_5)\\&-(-1)^{(|x_1|+|x_2|)(|x_3|+|x_5|)+|x_5|(\alpha+|x_3|+|x_4|)}\Phi(x_5,\mathcal{D}_q(x_3))\Omega(x_1,x_2,x_4),
	\end{split}
\end{equation}

\begin{equation}\label{e318}
	\begin{split}
		&(-1)^{\alpha(|x_1|+|x_2|+|x_3|)}\Omega([x_1,x_2,x_3]_\mathcal{Q},\mathcal{D}_q(x_4),x_5)\\
		&+(-1)^{|x_3|(|x_1|+|x_2|)+\alpha(|x_1|+|x_2|+|x_3|)}\Omega(x_3,[x_1,x_2,\mathcal{D}_q(x_4)]_\mathcal{Q},x_5)\\
		&+(-1)^{(|x_1|+|x_2|)(|x_3|+|x_4|)+\alpha|x_3|}\Omega (x_3,\mathcal{D}_q (x_4),[x_1,x_2,x_5]_\mathcal{Q})\\
		&-(-1)^{\alpha(|x_1|+|x_2|+|x_3|)}\Omega (x_1,x_2,[x_3,\mathcal{D}_q(x_4),x_5]_\mathcal{Q})+(-1)^{\alpha(|x_1|+|x_2|+|x_3|)}\Phi(x_1,x_2)\Omega(x_3,\mathcal{D}_q(x_4),x_5)\\
		&+(-1)^{(|x_3|+|x_5|)(\alpha+|x_1|+|x_2|)+|x_5|(|x_3|+|x_4|)+\alpha(|x_1|+|x_2|)}\Phi(x_5,x_3)\Omega(x_1,x_2,\mathcal{D}_q(x_4))\\
		&=-(-1)^{(|x_3|+|x_4|)(|x_1|+|x_2|)+\alpha|x_3|}\Phi(x_3,\mathcal{D}_q(x_4))\Omega(x_1,x_2,x_5)\\&-(-1)^{(|x_4|+|x_5|)(|x_1|+|x_2|+|x_3|)}\Phi(\mathcal{D}_q(x_4),x_5)\Omega(x_1,x_2,x_3),
	\end{split}
\end{equation}
and
\begin{equation}\label{e319}
	\begin{split}
		&(-1)^{\alpha(|x_1|+|x_2|+|x_3|+|x_4|)}\Omega([x_1,x_2,x_3]_\mathcal{Q},x_4,\mathcal{D}_q(x_5))\\
		&+(-1)^{|x_3|(|x_1|+|x_2|)+\alpha(|x_1|+|x_2|+|x_3|+|x_4|)}\Omega(x_3,[x_1,x_2,x_4]_\mathcal{Q},\mathcal{D}_q(x_5))\\
		&+(-1)^{(|x_1|+|x_2|)(|x_3|+|x_4|)+\alpha(|x_1|+|x_2|+|x_3|+|x_4|)}\Omega (x_3,x_4,[x_1,x_2,\mathcal{D}_q (x_5)]_\mathcal{Q})\\
		&-(-1)^{\alpha(|x_1|+|x_2|+|x_3|+|x_4|)}\Omega(x_1,x_2,[x_3,x_4,\mathcal{D}_q(x_5)]_\mathcal{Q})\\&+(-1)^{\alpha(|x_1|+|x_2|+|x_3|+|x_4|)}\Phi(x_1,x_2)\Omega(x_3,x_4,\mathcal{D}_q(x_5))\\&+(-1)^{(|x_3|+|x_4|)(\alpha+|x_1|+|x_2|)+\alpha(|x_1|+|x_2|)}\Phi(x_3,x_4)\Omega(x_1,x_2,\mathcal{D}_q(x_5))\\
		&=-(-1)^{(|x_4|+|x_5|)(|x_1|+|x_2|+|x_3|)}\Phi(x_4,\mathcal{D}_q(x_5))\Omega(x_1,x_2,x_3)\\&-(-1)^{(|x_1|+|x_2|)(|x_3|+|x_5|)+(|x_3|+|x_4|)(|x_5|+\alpha)}\Phi(\mathcal{D}_q(x_5),x_3)\Omega(x_1,x_2,x_4).
	\end{split}
\end{equation}

By inserting Eq \ref{e315}-\ref{e319} into Eq \ref{e34}, we get
\begin{equation*}\label{e320}
\begin{split}
		&(\delta_\Phi Ob^\Omega _{(\mathcal{D}_p,\mathcal{D}_q)})(x_1,x_2,x_3,x_4,x_5)\\
		&=-\mathcal{D}_p(\Phi(x_1,x_2)(\Omega(x_3,x_4,x_5)))+(-1)^{(|x_1|+|x_2|)(|x_3|+|x_4|)}\mathcal{D}_p(\Phi(x_3,x_4)(\Omega(x_1,x_2,x_5)))\\
		&+(-1)^{(|x_1|+|x_2|+|x_3|)(|x_4|+|x_5|)}\mathcal{D}_p(\Phi(x_4,x_5)(\Omega(x_1,x_2,x_3)))\\
		& +(-1)^{(|x_1|+|x_2|)(|x_3|+|x_5|)+|x_5|(|x_3|+|x_4|)}\mathcal{D}_p(\Phi(x_5,x_3)(\Omega(x_1,x_2,x_4)))\\  
		&+(-1)^{\alpha(|x_1|+|x_2|)}\Phi(x_1,x_2)(\mathcal{D}_p(\Omega(x_3,x_4,x_5)))\\
		&-(-1)^{(|x_3|+|x_4|)(\alpha+|x_1|+|x_2|)}\Phi(x_3,x_4)(\mathcal{D}_p(\Omega(x_1,x_2,x_5))\\
		&-(-1)^{(|x_4|+|x_5|)(\alpha+|x_1|+|x_2|+|x_3|)}\Phi(x_4,x_5)(\mathcal{D}_p(\Omega(x_1,x_2,x_3)))\\
		&-(-1)^{(|x_3|+|x_5|)(\alpha+|x_1|+|x_2|)+|x_5|(|x_3|+|x_4|)}\Phi(x_5,x_3)(\mathcal{D}_p(\Omega(x_1,x_2,x_4)))\\
		&+\Phi(\mathcal{D}_q(x_1),x_2)(\Omega(x_3,x_4,x_5))+ (-1)^{\alpha|x_1|}\Phi(x_1,\mathcal{D}_q(x_2))(\Omega(x_3,x_4,x_5))\\ &-(-1)^{(|x_1|+|x_2|)(|x_3|+|x_4|)}\Phi(\mathcal{D}_q(x_3),x_4)(\Omega(x_1,x_2,x_5))\\&-(-1)^{(|x_1|+|x_2|)(|x_3|+|x_5|)+|x_5|(\alpha+|x_3|+|x_4|)}\Phi(x_5,\mathcal{D}_q(x_3))(\Omega(x_1,x_2,x_4))\\
		&-(-1)^{(|x_3|+|x_4|)(|x_1|+|x_2|)+\alpha|x_3|}\Phi(x_3,\mathcal{D}_q(x_4))(\Omega(x_1,x_2,x_5))\\&-(-1)^{(|x_4|+|x_5|)(|x_1|+|x_2|+|x_3|)}\Phi(\mathcal{D}_q(x_4),x_5)(\Omega(x_1,x_2,x_3))\\&-(-1)^{(|x_4|+|x_5|)(|x_1|+|x_2|+|x_3|)+\alpha|x_4|}\Phi(x_4,\mathcal{D}_q(x_5))(\Omega(x_1,x_2,x_3))\\&-(-1)^{(|x_1|+|x_2|)(|x_3|+|x_5|)+|x_5|(|x_3|+|x_4|)}\Phi((\mathcal{D}_q(x_5),x_3)\Omega(x_1,x_2,x_4)
		\\&=-(\mathcal{D}_p\Phi(x_1,x_2)-(-1)^{\alpha(|x_1|+|x_2|)}\Phi(x_1,x_2)\mathcal{D}_p-\Phi(\mathcal{D}_q(x_1),x_2)\\&-(-1)^{\alpha|x_1|}\Phi(x_1,\mathcal{D}_q(x_2)))(\Omega(x_3,x_4,x_5))+(-1)^{(|x_1|+|x_2|)(|x_3|+|x_4|)}(\mathcal{D}_p\Phi(x_3,x_4)\\&-(-1)^{\alpha(|x_3|+|x_4|)}\Phi(x_3,x_4)\mathcal{D}_p-\Phi(\mathcal{D}_q(x_3),x_4)-(-1)^{\alpha|x_3|}\Phi(x_3,\mathcal{D}_q(x_4)))(\Omega(x_1,x_2,x_5))\\&
		+(-1)^{(|x_1|+|x_2|+|x_3|)(|x_4|+|x_5|)}(\mathcal{D}_p\Phi(x_4,x_5)-(-1)^{\alpha(|x_4|+|x_5|)}\Phi(x_4,x_5)\mathcal{D}_p-\Phi(\mathcal{D}_q(x_4),x_5)\\&-(-1)^{\alpha|x_4|}\Phi(x_4,\mathcal{D}_q(x_5)))(\Omega(x_1,x_2,x_3))
		+(-1)^{(|x_1|+|x_2|)(|x_3|+|x_5|)+|x_5|(|x_3|+|x_4|)}(\mathcal{D}_p\Phi(x_5,x_3)\\&-(-1)^{\alpha(|x_3|+|x_5|)}\Phi(x_5,x_3)\mathcal{D}_p-\Phi(\mathcal{D}_q(x_5),x_3)-(-1)^{\alpha|x_5|}\Phi(x_5,\mathcal{D}_q(x_3)))(\Omega(x_1,x_2,x_4))\\
		&=0.
		\end{split}
\end{equation*}
\end{proof}

\begin{definition}\label{d31}
	Let $(\Phi; \mathcal{P})$ be a representation of $\mathcal{Q}$. A pair $(\mathcal{D}_p,\mathcal{D}_q)\in Der(\mathcal{P})\times Der(\mathcal{Q})$ is called compatible (with respect to $\Phi$) if Eq \ref{e33} holds.
\end{definition}
Now we are ready to construct a Lie superalgebra and its representation on the first cohomology group. Set 
\begin{equation}\label{e321}
	\mathcal{T}_\Phi=\{(\mathcal{D}_p,\mathcal{D}_q)\in Der(\mathcal{P})\times Der(\mathcal{Q})|(\mathcal{D}_p,\mathcal{D}_q)~ is~ compatible~ with ~respect~ to~ \Phi\}.
\end{equation}
We have the following.
\begin{lemma}
	There is an even linear map $\Psi:\mathcal{T}_\Phi\rightarrow gl(\mathbb{H}^1(\mathcal{Q};\mathcal{P}))$ given by 
	\begin{equation}\label{e322}
		\Psi(\mathcal{D}_p,\mathcal{D}_q)([\Phi])=[Ob^\Omega _{(\mathcal{D}_p,\mathcal{D}_q)}]~~~~~for~\Omega\in (\mathbb{Z}^1(\mathcal{Q};\mathcal{P}))_{\overline{0}},
	\end{equation}
	where $[Ob^\Omega _{(\mathcal{D}_p,\mathcal{D}_q)}]$ is given by Eq \ref{e32}.
\end{lemma}
\begin{proof}
	Since $(\mathcal{D}_p,\mathcal{D}_q)$ is compatible with respect to $\Phi$, in Lemma \ref{l31} it is proved that $Ob^\Omega _{(\mathcal{D}_p,\mathcal{D}_q)}$ is a 2-cocycle whenever $\Omega$ is a 2-cocycle. So it is sufficient to show that if $\delta_{\Phi} \lambda$ is a 2-coboundary then $\Phi(\mathcal{D}_p,\mathcal{D}_q)(\delta_{\Phi} \lambda)=0$ which implies that $\Psi$ is well-defined. Here $|\lambda|=0$ and $|\mathcal{D}_p|=|\mathcal{D}_q|=\alpha$.
	\begin{equation*}
	\begin{split}
			&(\Psi(\mathcal{D}_p,\mathcal{D}_q)(\delta_{\Phi} \lambda))(x,y,z)\\
			&=\mathcal{D}_p(\delta_{\Phi} \lambda)(x,y,z)-(\delta_{\Phi} \lambda)(\mathcal{D}_q(x),y,z)-(-1)^{\alpha|x|}(\delta_{\Phi} \lambda)(x,\mathcal{D}_q(y),z)\\
			&~~~~~~-(-1)^{\alpha(|x|+|y|)}(\delta_{\Phi} \lambda)(x,y,\mathcal{D}_q(z))\\
			&=\mathcal{D}_p(-\lambda([x,y,z]_\mathcal{Q})+\Phi(x,y)(\lambda(z))+(-1)^{|x|(|y|+|z|)}\Phi(y,z)(\lambda(x))\\
			&~~~~~~+(-1)^{|z|(|x|+|y|)}\Phi(z,x)(\lambda(y)))-(-\lambda([\mathcal{D}_q(x),y,z]_\mathcal{Q})+\Phi(\mathcal{D}_q(x),y)(\lambda(z))\\
			&~~~~~~+(-1)^{(\alpha+|x|)(|y|+|z|)}\Phi(y,z)(\lambda(\mathcal{D}_q(x)))+(-1)^{|z|(|x|+|y|+\alpha)}\Phi(z,\mathcal{D}_q(x))(\lambda(y)))\\
			&~~~~~~+(-1)^{\alpha|x|}(-\lambda([x,\mathcal{D}_q(y),z]_\mathcal{Q})+\Phi(x,\mathcal{D}_q(y))(\lambda(z))+(-1)^{|x|(\alpha+|y|+|z|)}\Phi(\mathcal{D}_q(y),z)(\lambda(x))\\
			&~~~~~~+(-1)^{|z|(|x|+|y|+\alpha)}\Phi(z,x)(\lambda(\mathcal{D}_q(y))))-(-1)^{\alpha(|x|+|y|)} (-\lambda([x,y,\mathcal{D}_q(z)]_\mathcal{Q})+\Phi(x,y)(\lambda(\mathcal{D}_q(z)))\\
			&~~~~~~+(-1)^{|x|(\alpha+|y|+|z|)}\Phi(y,\mathcal{D}_q(z))(\lambda(x))+(-1)^{(|z|+\alpha)(|x|+|y|)}\Phi(\mathcal{D}_q(z),x)(\lambda(y))).
	\end{split}
	\end{equation*}
Since $\mathcal{D}_q$ is a superderivation, 
\begin{equation*}
	\begin{split}
		&\lambda([\mathcal{D}_q(x),y,z]_{\mathcal{Q}})+(-1)^{\alpha|x|}\lambda([x,\mathcal{D}_q(y),z]_{\mathcal{Q}}))+(-1)^{\alpha(|x|+|y|)}  \lambda([x,y,\mathcal{D}_q(z)]_{\mathcal{Q}})\\&=\lambda(\mathcal{D}_q([x,y,z]_{\mathcal{Q}})).
	\end{split}
\end{equation*}
Then we have 
\begin{equation*}
	\begin{split}
		&(\Psi(\mathcal{D}_p,\mathcal{D}_q)(\delta_{\Phi} \lambda))(x,y,z)\\
		&=(\mathcal{D}_p\Phi(x,y)(\lambda(z))-\Phi(\mathcal{D}_q(x),y)(\lambda(z))-(-1)^{\alpha|x|}\Phi(x,\mathcal{D}_q(y))(\lambda(z)))\\
		&+(-1)^{\alpha(|x|+|y|)} (\mathcal{D}_p\Phi(y,z)(\lambda(x))-\Phi(\mathcal{D}_q(y),z)(\lambda(x))-(-1)^{\alpha|y|}\Phi(y,\mathcal{D}_q(z))(\lambda(x)))\\
		&+(-1)^{|z|(|x|+|y|)} (\mathcal{D}_p\Phi(z,x)(\lambda(y))-\Phi(\mathcal{D}_q(z),x)(\lambda(y))-(-1)^{\alpha|z|}\Phi(x,\mathcal{D}_q(x))(\lambda(y)))\\
		&-(-1)^{(\alpha+|x|)(|y|+|z|)}\Phi(y,z)(\lambda(\mathcal{D}_q(x)))+(-1)^{\alpha|x|+|z|(|x|+|y|+\alpha)}\Phi(z,x)(\lambda(\mathcal{D}_q(y)))\\
		&-(-1)^{\alpha(|x|+|y|)} \Phi(x,y)(\lambda(\mathcal{D}_q(z)))-\mathcal{D}_p(\lambda([x,y,z]_{\mathcal{Q}}))+\lambda(\mathcal{D}_q([x,y,z]_{\mathcal{Q}})).
	\end{split}
\end{equation*}
By using Eq \ref{e33}, we have
\begin{equation*}
	\begin{split}
		&\mathcal{D}_p\Phi(x,y)(\lambda(z))-\Phi(\mathcal{D}_q(x),y)(\lambda(z))-(-1)^{\alpha|x|}\Phi(x,\mathcal{D}_q(y))(\lambda(z))\\&=(-1)^{\alpha(|x|+|y|)}\Phi(x,y)\mathcal{D}_p(\lambda(z)),
	\end{split}
\end{equation*}
\begin{equation*}
	\begin{split}
		&(-1)^{|x|(|y|+|z|)}(\mathcal{D}_p\Phi(y,z)(\lambda(x))-\Phi(\mathcal{D}_q(y),z)(\lambda(x))-(-1)^{\alpha||y|}\Phi(y,\mathcal{D}_q(z))(\lambda(x)))\\&=(-1)^{|x|(|y|+|z|)+\alpha(|y|+|z|)}\Phi(y,z)\mathcal{D}_p(\lambda(x)),
		\end{split}\end{equation*}

\begin{equation*}
	\begin{split}
		&(-1)^{|z|(|x|+|y|)}(\mathcal{D}_p\Phi(z,x)(\lambda(y))-\Phi(\mathcal{D}_q(z),x)(\lambda(y))-(-1)^{\alpha|z|}\Phi(z,\mathcal{D}_q(x))(\lambda(y)))\\&=(-1)^{|z|(|x|+|y|)+\alpha(|z|+|x|)}\Phi(z,x)\mathcal{D}_p(\lambda(y)).
		\end{split}
\end{equation*}

\begin{equation}
	\begin{split}
		&(\Psi(\mathcal{D}_p,\mathcal{D}_q)(\delta_{\Phi} \lambda))(x,y,z)\\
		&=(-1)^{\alpha(|x|+|y|)}\Phi(x,y)\mathcal{D}_p(\lambda(z))+(-1)^{(\alpha+|x|)(|y|+|z|)}\Phi(y,z)\mathcal{D}_p(\lambda(x))\\&+(-1)^{|z|(|x|+|y|)+\alpha(|z|+|x|)}\Phi(z,x)\mathcal{D}_p(\lambda(y))-(-1)^{(\alpha+|x|)(|y|+|z|)}\Phi(y,z)(\lambda(\mathcal{D}_q(x)))\\&-(-1)^{|z|(|x|+|y|)+\alpha(|z|+|x|)}\Phi(z,x)(\mathcal{D}_q(\lambda(y)))-(-1)^{\alpha(|x|+|y|)} \Phi(x,y)(\mathcal{D}_q(\lambda(z)))\\
		& -\mathcal{D}_p(\lambda([x,y,z]_{\mathcal{Q}}))+\lambda(\mathcal{D}_q([x,y,z]_{\mathcal{Q}}))\\
		&=\delta_{\Phi}(\mathcal{D}_p \lambda-\lambda \mathcal{D}_q)(x,y,z),
	\end{split}
\end{equation}
which implies that $[Ob^{\Omega,1} _{(\mathcal{D}_p,\mathcal{D}_q)}]=[Ob^{\Omega,2} _{(\mathcal{D}_p,\mathcal{D}_q)}]\in \mathbb{H}^1(\mathcal{Q},\mathcal{P})$ as required.
\end{proof}

Below is the main result of this section.
\begin{theorem}
	Keep notation as above. For any representation $(\Phi;\mathcal{P})$ of $\mathcal{Q}$, $\mathcal{T}_\Phi$ is a Lie subsuperalgebra of $Der(\mathcal{P})\times Der(\mathcal{Q})$ and the map $\Psi$ given by Eq \ref{e322} is a Lie superalgebra homomorphism.
\end{theorem}
\begin{proof}
	Take $|\mathcal{D}_{p_1}|=\alpha_1$ and $|\mathcal{D}_{p_2}|=\alpha_2$.
	\begin{equation}\label{e323}
		\begin{split}
			&(\mathcal{D}_{p_1}\mathcal{D}_{p_2}-(-1)^{\alpha_1\alpha_2}\mathcal{D}_{p_2}\mathcal{D}_{p_1})\Phi (x,y)- (-1)^{(\alpha_1+\alpha_2)(|x|+|y|)} \Phi (x,y)(\mathcal{D}_{p_1}\mathcal{D}_{p_2}-(-1)^{\alpha_1\alpha_2}\mathcal{D}_{p_2}\mathcal{D}_{p_1})\\
			&=\underbrace{\mathcal{D}_{p_1}((-1)^{\alpha_2 (|x|+|y|)}\Phi(x,y)\mathcal{D}_{p_2}+\Phi (\mathcal{D}_{p_2}(x),y)+(-1)^{\alpha_2|x|}\Phi (x,\mathcal{D}_{p_2}(y)))}_{I_1}\\
			&-\underbrace{(-1)^{\alpha_1\alpha_2}\mathcal{D}_{p_2}((-1)^{\alpha_1 (|x|+|y|)}\Phi(x,y)\mathcal{D}_{p_1}+\Phi (\mathcal{D}_{p_1}(x),y)+(-1)^{\alpha_1|x|}\Phi (x,\mathcal{D}_{p_1}(y)))}_{I_2}\\
			&-(-1)^{(\alpha_1+\alpha_2) (|x|+|y|)}\Phi(x,y)\mathcal{D}_{p_1}\mathcal{D}_{p_2} (-1)^{(\alpha_1+\alpha_2) (x+y)+\alpha_1\alpha_2}\Phi(x,y)\mathcal{D}_{p_2}\mathcal{D}_{p_1},
		\end{split}
	\end{equation}
where
\begin{equation*}\label{e324}
	I_1=\mathcal{D}_{p_1}((-1)^{\alpha_2 (|x|+|y|)}\Phi(x,y)\mathcal{D}_{p_2}+\Phi (\mathcal{D}_{p_2}(x),y)+(-1)^{\alpha_2|x|}\Phi (x,\mathcal{D}_{p_2}(y))),      
\end{equation*}
\begin{equation*}\label{e325}
	I_2=(-1)^{\alpha_1\alpha_2}\mathcal{D}_{p_2}((-1)^{\alpha_1 (|x|+|y|)}\Phi(x,y)\mathcal{D}_{p_1}+\Phi (\mathcal{D}_{p_1}(x),y)+(-1)^{\alpha_1 |x|}\Phi (x,\mathcal{D}_{p_1}(y))).
\end{equation*}
By Eq \ref{e33}, we get
\begin{equation}\label{e326}
\begin{split}
		&I_1=(-1)^{\alpha_2(|x|+|y|)}((-1)^{\alpha_1(|x|+|y|)}	\Phi(x,y)\mathcal{D}_{p_1}+\Phi (\mathcal{D}_{p_1}(x),y)\\
		&+(-1)^{\alpha_1|x|}\Phi (x,\mathcal{D}_{p_1}(y)))\mathcal{D}_{p_2}+((-1)^{\alpha_1 (|x|+|y|+\alpha_2)}\Phi(\mathcal{D}_{p_2}(x),y)\mathcal{D}_{p_1}\\&+\Phi (\mathcal{D}_{p_1}\mathcal{D}_{p_2}(x),y)+(-1)^{\alpha_1(|\alpha_2|+|x|)}\Phi (\mathcal{D}_{p_2}(x),\mathcal{D}_{p_1}(y)))\\&+(-1)^{\alpha_2|x|}((-1)^{\alpha_1 (|x|+|y|+\alpha_2)}\Phi(x,\mathcal{D}_{p_2}(y))\mathcal{D}_{p_1}+\Phi (\mathcal{D}_{p_1}(x),\mathcal{D}_{p_2}(y)) \\&+(-1)^{\alpha_1|x|}\Phi (x,\mathcal{D}_{p_1}\mathcal{D}_{p_2}(y))),
\end{split}
\end{equation}
and
\begin{equation}\label{e327}
	\begin{split}
		&I_2=-(-1)^{\alpha_1(\alpha_2+|x|+|y|)}((-1)^{\alpha_2(|x|+|y|)}	\Phi(x,y)\mathcal{D}_{p_2}+\Phi (\mathcal{D}_{p_2}(x),y)\\&+(-1)^{\alpha_2|x|}\Phi (x,\mathcal{D}_{p_2}(y)))\mathcal{D}_{p_1}-(-1)^{\alpha_1\alpha_2}((-1)^{\alpha_2 (|x|+|y|+\alpha_1)}\Phi(\mathcal{D}_{p_1}(x),y)\mathcal{D}_{p_2}\\&+\Phi (\mathcal{D}_{p_2}\mathcal{D}_{p_1}(x),y)+(-1)^{\alpha_2(\alpha_1+|x|)}\Phi (\mathcal{D}_{p_1}(x),\mathcal{D}_{p_2}(y)))\\&-(-1)^{\alpha_1(|x|+\alpha_2)}((-1)^{\alpha_2 (|x|+|y|+\alpha_1)}\Phi(x,\mathcal{D}_{p_1}(y))\mathcal{D}_{p_2}+\Phi (\mathcal{D}_{p_2}(x),\mathcal{D}_{p_1}(y)) \\&+(-1)^{\alpha_2|x|}\Phi (x,\mathcal{D}_{p_2}\mathcal{D}_{p_1}(y))).
	\end{split}
\end{equation}
Then inserting Eqs \ref{e326} and \ref{e327} into Eq \ref{e323}, we get
\begin{equation}\label{e328}
	\begin{split}
		&(\mathcal{D}_{p_1}\mathcal{D}_{p_2}-(-1)^{\alpha_1\alpha_2}\mathcal{D}_{p_2}\mathcal{D}_{p_1})\Phi (x,y)- (-1)^{(\alpha_1+\alpha_2)(|x|+|y|)} \Phi (x,y)(\mathcal{D}_{p_1}\mathcal{D}_{p_2}\\
		&-(-1)^{\alpha_1\alpha_2}\mathcal{D}_{p_2}\mathcal{D}_{p_1})=\Phi((\mathcal{D}_{q_1}\mathcal{D}_{q_2}-(-1)^{\alpha_1\alpha_2|}\mathcal{D}_{q_2}\mathcal{D}_{q_1})(x),y)\\
		&+ (-1)^{|x|(\alpha_1+\alpha_2)}\Phi (x,(\mathcal{D}_{q_1}\mathcal{D}_{q_2}-(-1)^{\alpha_1\alpha_2}\mathcal{D}_{q_2}\mathcal{D}_{q_1})(y)),
	\end{split}
\end{equation}
which implies that $[(\mathcal{D}_{p_1},\mathcal{D}_{q_1}),(\mathcal{D}_{p_2},\mathcal{D}_{q_2})]$ is compatible by Definition \ref{d31}.\\
\noindent To prove second part refer \cite{xu2018}.
\end{proof}

\section{Abelian Extensions and Extensibility of superderivations}
In this section, we construct obstruction classes for extensibility of superderivations by Lemma \ref{l31} and also give a representation of $\mathcal{T}_{\Phi}$ in terms of extensibility of superderivations.

\begin{lemma}\label{l41}
	Keep notations as above. The cohomology class $[Ob^\Omega _{(\mathcal{D}_p,\mathcal{D}_q)}]\in \mathbb{H}^1(\mathcal{Q};\mathcal{P})$ does not depend on the choice of the section $\pi$.
\end{lemma}

\begin{proof}Let $s_1$ and $s_2$ be sections of $\pi$ and $\Omega_1,\Omega_2$ be defined by Eq \ref{e29} while $Ob^{\Omega,1} _{(\mathcal{D}_p,\mathcal{D}_q)},~Ob^{\Omega,2} _{(\mathcal{D}_p,\mathcal{D}_q)}$ are defined by Eq \ref{e32} with respect to $\Omega_1,\Omega_2$. Then
	\begin{equation}\label{e41}
		\begin{split}
				&Ob^{\Omega,1} _{(\mathcal{D}_p,\mathcal{D}_q)}(x,y,z)-	Ob^{\Omega,2} _{(\mathcal{D}_p,\mathcal{D}_q)}(x,y,z)\\
				&=\mathcal{D}_p(\Omega_1(x,y,z))-\Omega_1(\mathcal{D}_q(x),y,z)-(-1)^{\alpha |x|}\Omega_1(x,\mathcal{D}_q(y),z)-(-1)^{\alpha (|x|+|y|)}\Omega_1(x,y,\mathcal{D}_q(z))\\
				&~~~~~-\mathcal{D}_p(\Omega_2(x,y,z))+\Omega_2(\mathcal{D}_q(x),y,z)+(-1)^{\alpha |x|}\Omega_2(x,\mathcal{D}_q(y),z)+(-1)^{\alpha (|x|+|y|)}\Omega_2(x,y,\mathcal{D}_q(z))\\
				&=\underbrace{\mathcal{D}_p(\Omega_1(x,y,z)-\Omega_2(x,y,z))}_{I_1}-\underbrace{(\Omega_1(\mathcal{D}_q(x),y,z)-\Omega_2(\mathcal{D}_q(x),y,z))}_{I_2}\\
				&~~~~~-(-1)^{\alpha |x|}\underbrace{(\Omega_1(x,\mathcal{D}_q(y),z)-\Omega_2(x,\mathcal{D}_q(y),z))}_{I_3}\\
				&~~~~~-(-1)^{\alpha (|x|+|y|)}\underbrace{(\Omega_1(x,y,\mathcal{D}_q(z))-\Omega_2(x,y,\mathcal{D}_q(z)))}_{I_4},
		\end{split}
	\end{equation}
for $x,y,z\in \mathcal{Q}$ and $|\mathcal{D}_p|=|\mathcal{D}_q|=\alpha$.
Define an even linear map $\lambda:\mathcal{Q}\rightarrow \mathcal{P}$ by $\lambda(x)=s_1(x)-s_2(x)$ where $x\in \mathcal{Q}$. From Corollary \ref{coro21}, we have 
\begin{equation}\label{e42}
	\begin{split}
		\Omega_1(x,y,z)-\Omega_2(x,y,z)=&-\lambda([x,y,z]_\mathcal{Q})+\Phi(x,y)(\lambda(z))+(-1)^{|x|(|y|+|z|)}\Phi(y,z)(\lambda(x))\\
		&~~~~~~~+(-1)^{|z|(|x|+|y|)}\Phi(z,x)(\lambda(y)),
	\end{split}
\end{equation}
for any $x,y,z\in \mathcal{Q}$. Therefore, we get 
\begin{equation*} \begin{split}
		I_1&=\mathcal{D}_p(\Omega_1(x,y,z)-\Omega_2(x,y,z))\\
		&=\mathcal{D}_p(-\lambda([x,y,z]_\mathcal{Q})+\Phi(x,y)(\lambda(z))+(-1)^{|x| (|y|+|z|)}\Phi(y,z)(\lambda(x))\\&+(-1)^{|z|(|x|+|y|)}\Phi(z,x)(\lambda(y))),
\end{split} \end{equation*}

\begin{equation*}
	\begin{split}
		I_2&=\Omega_1(\mathcal{D}_q(x),y,z)-\Omega_2(\mathcal{D}_q(x),y,z)\\
		&=-\lambda([\mathcal{D}_q(x),y,z]_\mathcal{Q})+\Phi(\mathcal{D}_q(x),y)(\lambda(z))+(-1)^{(\alpha+|x|) (|y|+|z|)}\Phi(y,z)(\lambda(\mathcal{D}_q(x)))\\
		&+(-1)^{|z| (\alpha+|x|+|y|)}\Phi(z,\mathcal{D}_q(x))(\lambda(y)),
		\end{split}
\end{equation*}

\begin{equation*}
	\begin{split}
		I_3&=\Omega_1(x,\mathcal{D}_q(y),z)-\Omega_2(x,\mathcal{D}_q(y),z)\\
		&=-\lambda([x,\mathcal{D}_q(y),z]_\mathcal{Q})+\Phi(x,\mathcal{D}_q(y))(\lambda(z))+(-1)^{|x|(\alpha+|y|+|z|)}\Phi(\mathcal{D}_q(y),z)(\lambda(x))\\
		&~~~~~+(-1)^{|z|(\alpha+|x|+|y|)}\Phi(z,x)(\lambda(\mathcal{D}_q(y))),
		\end{split}
\end{equation*} and
\begin{equation*}
	\begin{split}
			I_4&=\Omega_1(x,y,\mathcal{D}_q(z))-\Omega_2(x,y,\mathcal{D}_q(z))\\
		&=-\lambda([x,y,\mathcal{D}_q(z)]_\mathcal{Q})+\Phi(x,y)(\lambda(\mathcal{D}_q(z)))+(-1)^{|z| (\alpha+|x|+|y|)}\Phi(y,\mathcal{D}_q(z))(\lambda(x))\\
		&+(-1)^{(\alpha+|z|) (|x|+|y|)}\Phi(\mathcal{D}_q(z),x)(\lambda(y)).
		\end{split}
\end{equation*}

By $I_1,I_2,I_3,I_4$, and Eq \ref{e41}, we get
\begin{equation*}
	\begin{split}
		&Ob^{\Omega,1} _{(\mathcal{D}_p,\mathcal{D}_q)}(x,y,z)-	Ob^{\Omega,2} _{(\mathcal{D}_p,\mathcal{D}_q)}(x,y,z)\\
		&=(\mathcal{D}_p\Phi(x,y)(\lambda(z))-\Phi(\mathcal{D}_q(x),y)(\lambda(z))-(-1)^{\alpha |x|}\Phi(x,\mathcal{D}_q(y))(\lambda(z)))\\
		&~~~~~~+(-1)^{|x|(|y|+|z|)}(\mathcal{D}_p\Phi(y,z)(\lambda(x))-\Phi(\mathcal{D}_q(y),z)(\lambda(x))-(-1)^{\alpha |y|}\Phi(y,\mathcal{D}_q(z))(\lambda(x)))
		\\
		&~~~~~~+(-1)^{|z|(|x|+|y|)}(\mathcal{D}_p\Phi(z,x)(\lambda(y))-\Phi(\mathcal{D}_q(z),x)(\lambda(y))-(-1)^{\alpha|z|}\Phi(z,\mathcal{D}_q(x))(\lambda(y)))\\
		&~~~~~~-\mathcal{D}_p(\lambda([x,y,z]_\mathcal{Q}))-(-1)^{(\alpha+|x|)(|y|+|z|)}\Phi(y,z)\lambda(\mathcal{D}_q(x))\\
		&~~~~~~-(-1)^{\alpha|x|+|z|(\alpha+|x|+|y|)}\Phi(z,x)\lambda(\mathcal{D}_q(y))-(-1)^{\alpha(|x|+|y|)}\Phi(x,y)\lambda(\mathcal{D}_q(z))+\lambda(\mathcal{D}_q([x,y,z]_\mathcal{Q})),
	\end{split}
\end{equation*}
where $\mathcal{D}_q$ is a superderivation. Since $(\mathcal{D}_p,\mathcal{D}_q)$ is compatible, by Eq \ref{e33} it follows that
\begin{equation*}
	\begin{split}
		&(\mathcal{D}_p\Phi(x,y)(\lambda(z))-\Phi(\mathcal{D}_q(x),y)(\lambda(z))-(-1)^{\alpha |x|}\Phi(x,\mathcal{D}_q(y))(\lambda(z)))\\&=(-1)^{\alpha(|x|+|y|)}\Phi(x,y)(\mathcal{D}_p(\lambda(z))),
	\end{split}
\end{equation*}
\begin{equation*}
	\begin{split}
		(-1)^{|x|(|y|+|z|)}(\mathcal{D}_p\Phi(y,z)(\lambda(x))&-\Phi(\mathcal{D}_q(y),z)(\lambda(x))-(-1)^{\alpha |y|}\Phi(y,\mathcal{D}_q(z))(\lambda(x)))\\&=(-1)^{(|x|+\alpha)(|y|+|z|)}\Phi(y,z)(\mathcal{D}_p(\lambda(x))),
	\end{split}
\end{equation*}
and
\begin{equation*}
	\begin{split}
		(-1)^{|z|(|x|+|y|)}(\mathcal{D}_p\Phi(z,x)(\lambda(y))&-\Phi(\mathcal{D}_q(z),x)(\lambda(y))-(-1)^{\alpha |z|}\Phi(z,\mathcal{D}_q(x))(\lambda(y)))\\&=(-1)^{|z|(|x|+|y|+\alpha)+\alpha|x|}\Phi(z,x)\mathcal{D}_p(\lambda(y)).
	\end{split}
\end{equation*}
Therefore we obtain
\begin{equation*}
\begin{split}
	&Ob^{\Omega,1} _{(\mathcal{D}_p,\mathcal{D}_q)}(x,y,z)-	Ob^{\Omega,2} _{(\mathcal{D}_p,\mathcal{D}_q)}(x,y,z)\\
		&=(-1)^{\alpha(|x|+|y|)}\Phi(x,y)(\mathcal{D}_p(\lambda(z)))-\mathcal{D}_p(\lambda([x,y,z]_\mathcal{Q}))+(-1)^{(|x|+\alpha)(|y|+|z|)}\Phi(y,z)(\mathcal{D}_p(\lambda(x)))\\
		&~~~~~~+(-1)^{|z|(|x|+|y|+\alpha)+\alpha|x|}\Phi(z,x)(\mathcal{D}_p(\lambda(y)))+(-1)^{(\alpha+|x|)(|y|+|z|)}\Phi(y,z)(\lambda(\mathcal{D}_q(x)))\\
		&~~~~~~-(-1)^{|z|(\alpha+|x|+|y|)+\alpha|x|}\Phi(z,x)(\lambda(\mathcal{D}_q(y)))-(-1)^{\alpha(|x|+|y|)}\Phi(x,y)(\lambda(\mathcal{D}_q(z)))+\lambda(\mathcal{D}_q([x,y,z]_\mathcal{Q}))\\
		&=\delta_{\Phi}(\mathcal{D}_p \lambda-\lambda   \mathcal{D}_q)(x,y,z),
		\end{split}
\end{equation*}
which implies that $[Ob^{\Omega,1} _{(\mathcal{D}_p,\mathcal{D}_q)}]=[Ob^{\Omega,2} _{(\mathcal{D}_p,\mathcal{D}_q)}]\in \mathbb{H}^1(\mathcal{Q};\mathcal{P})$ as required.
\end{proof}
\par Now we will define extensibility of superderivations.
\begin{definition}
	Let $0\rightarrow \mathcal{P}\hookrightarrow \mathcal{L}\rightarrow \mathcal{Q}\rightarrow 0$ be an extension of $3$-Lie superalgebras with $[\mathcal{P},\mathcal{P},\mathcal{L}]=0$. A pair $(\mathcal{D}_p,\mathcal{D}_q)\in Der(\mathcal{P})\times Der(\mathcal{Q})$ is called extensible if there is a superderivation $\mathcal{D}_l\in Der(\mathcal{L})$ such that the following diagram:
	\begin{center}
		\begin{tikzpicture}[>=latex]\label{diagram}
			\node (A_{1}) at (0,0) {\(0\)};
			\node (A_{2}) at (2,0) {\(\mathcal{P}\)};
			\node (A_{3}) at (4,0) {\(\mathcal{L}\)};
			\node (A_{4}) at (6,0) {\(\mathcal{Q}\)};
			\node (A_{5}) at (8,0) {\(0\)};
			\node (B_{1}) at (0,-2) {\(0\)};
			\node (B_{2}) at (2,-2) {\(\mathcal{P}\)};
			\node (B_{3}) at (4,-2) {\(\mathcal{L}\)};
			\node (B_{4}) at (6,-2) {\(\mathcal{Q}\)};
			\node (B_{5}) at (8,-2) {\(0\)};
			\draw[->] (A_{1}) -- (A_{2});
			\draw[->] (A_{2}) -- (A_{3}) node[midway,above] {$i$};
			\draw[->] (A_{3}) -- (A_{4}) node[midway,above] {$\pi$};
			\draw[->] (A_{4}) -- (A_{5});
			\draw[->] (B_{1}) -- (B_{2});
			\draw[->] (B_{2}) -- (B_{3}) node[midway,above] {$i$}; 
			\draw[->] (B_{3}) -- (B_{4}) node[midway,above] {$\pi$};
			\draw[->] (B_{4}) -- (B_{5});
			\draw[->] (A_{2}) -- (B_{2}) node[midway,right] {$\mathcal{D}_p$};
			\draw[->] (A_{3}) -- (B_{3}) node[midway,right] {$\mathcal{D}_l$};
			\draw[->] (A_{4}) -- (B_{4}) node[midway,right] {$\mathcal{D}_q$};
		\end{tikzpicture}\\
	\end{center}

commutes, where $i:\mathcal{P}\rightarrow \mathcal{L}$ is the inclusion map.
\end{definition}
The following result means that extensibility implies compatibility.
\begin{proposition}\label{p42}
	Let $0\rightarrow \mathcal{P}\hookrightarrow \mathcal{L}\rightarrow \mathcal{Q}\rightarrow 0$ be an extension of $3$-Lie superalgebras with $[\mathcal{P},\mathcal{P},\mathcal{L}]=0$. If $(\mathcal{D}_p,\mathcal{D}_q)\in Der(\mathcal{P})\times Der(\mathcal{Q})$ is extensible, then $(\mathcal{D}_p,\mathcal{D}_q)$ is compatible with respect to $\Phi$ given by Eq \ref{e27}.
\end{proposition}

\begin{proof}
Since $(\mathcal{D}_p,\mathcal{D}_q)$ is extensible, there exists a superderivation $\mathcal{D}_l\in Der(\mathcal{L})$ such that $i\mathcal{D}_p=\mathcal{D}_li,~\pi \mathcal{D}_l=\mathcal{D}_q\pi$, and $|\mathcal{D}_p|=|\mathcal{D}_q|=|\mathcal{D}_l|=\alpha$. Then 
$$\mathcal{D}_ls(x)-s(\mathcal{D}_q(x))\in \mathcal{P}~ {\rm{for}} ~x\in \mathcal{Q}.$$ So there is an even linear map $\mu:\mathcal{Q}\rightarrow \mathcal{P}$ given by 
\begin{equation}\label{e43}
	\mu(x)=\mathcal{D}_ls(x)-s(\mathcal{D}_q(x)).
\end{equation}
Since $[\mathcal{P},\mathcal{P},\mathcal{L}]=0$, we have 
\begin{equation*}
	[\mu(x),s(y),v]_\mathcal{L}=[s(x),\mu(y),v]_\mathcal{L}=0,
\end{equation*}
for $x,y\in \mathcal{Q}$ and $v\in \mathcal{P}$. Since $i\mathcal{D}_p=\mathcal{D}_l i$ and $\mathcal{D}_l\in Der(\mathcal{L})$, we get
\begin{equation*}
	\begin{split}
		&\mathcal{D}_p(\Phi(x,y)(v))-(-1)^{\alpha(|x|+|y|)}\Phi(x,y)(\mathcal{D}_p(v))\\
		&=\mathcal{D}_p([s(x),s(y),v]_\mathcal{L})-(-1)^{\alpha(|x|+|y|)}[s(x),s(y),\mathcal{D}_p(v)]_\mathcal{L}\\
		&=\mathcal{D}_l([s(x),s(y),v]_\mathcal{L})-(-1)^{\alpha(|x|+|y|)}[s(x),s(y),\mathcal{D}_p(v)]_\mathcal{L}\\
		&=[\mathcal{D}_l(s(x)),s(y),v]_\mathcal{L}+(-1)^{\alpha|x|}[s(x),\mathcal{D}_l(s(y)),v]_\mathcal{L}+(-1)^{\alpha(|x|+|y|)}[s(x),s(y),\mathcal{D}_l(v)]_\mathcal{L}\\
		&~~~~~~-(-1)^{\alpha(|x|+|y|)}[s(x),s(y),\mathcal{D}_p(v)]_\mathcal{L}\\
		&=[s(\mathcal{D}_q(x)),s(y),v]_\mathcal{L}+[\mu(x),s(y),v]_\mathcal{L}+(-1)^{\alpha|x|}([s(x),s(\mathcal{D}_q(y)),v]_\mathcal{L}\\
		&~~~~~~+[s(x),\mu(y),v]_\mathcal{L})+(-1)^{\alpha(|x|+|y|)}[s(x),s(y),\mathcal{D}_l(v)]_\mathcal{L}-(-1)^{\alpha(|x|+|y|)}[s(x),s(y),\mathcal{D}_p(v)]_\mathcal{L}\\
		&=[s(\mathcal{D}_q(x)),s(y),v]_\mathcal{L}+(-1)^{\alpha|x|}([s(x),s(\mathcal{D}_q(y)),v]_\mathcal{L}\\
		&=\Phi(\mathcal{D}_q(x),y)(v)+(-1)^{\alpha|x|}\Phi(x,\mathcal{D}_q(y))(v).
	\end{split}
\end{equation*}
\end{proof}

\begin{proposition}\label{p43}
	Let $0\rightarrow \mathcal{P}\hookrightarrow \mathcal{L}\rightarrow \mathcal{Q}\rightarrow 0$ be an extension of $3$-Lie superalgebras with $[\mathcal{P},\mathcal{P},\mathcal{L}]=0$. Assume that  $(\mathcal{D}_p,\mathcal{D}_q)\in Der(\mathcal{P}_{\overline{0}})\times Der(\mathcal{Q}_{\overline{0}})$ is compatible with respect to $\Phi$ given by Eq \ref{e27}. Then $(\mathcal{D}_p,\mathcal{D}_q)_{\overline{0}}$ is extensible if and only if $[Ob^\mathcal{L}_{(\mathcal{D}_p,\mathcal{D}_q)}]\in \mathbb{H}^1(\mathcal{Q};\mathcal{P})$ is trivial.
\end{proposition}
\begin{proof}
	Suppose that $(\mathcal{D}_p,\mathcal{D}_q)$ is extensible. Then there exists a superderivation $\mathcal{D}_l\in Der(\mathcal{L})$ such that the associative diagram \ref{diagram} is commutative. Since $\pi \mathcal{D}_l=\mathcal{D}_q  \pi $ and $|\mathcal{D}_p|=|\mathcal{D}_q|=|\mathcal{D}_l|=\alpha$, we have $$\mathcal{D}_ls(x)-s(\mathcal{D}_q(x))\in \mathcal{P}~ {\rm{for}} ~x\in \mathcal{Q}.$$ So there is an even linear map $\mu:\mathcal{Q}\rightarrow \mathcal{P}$ given by Eq \ref{e43}. It is sufficient to show that 
	\begin{equation}\label{e44}
Ob^\mathcal{L}_{(\mathcal{D}_p,\mathcal{D}_q)}(x_1,x_2,x_3)=(\delta_{\Phi}(\mu))(x_1,x_2,x_3),~~~~x_i\in \mathcal{Q}.
	\end{equation}
Now
\begin{equation}\label{e45}
\begin{split}
		&\mathcal{D}_l([s(x_1)+v_1,s(x_2)+v_2,s(x_3)+v_3]_\mathcal{L})\\
	&=[\mathcal{D}_l(s(x_1)+v_1)),s(x_2)+v_2,s(x_3)+v_3]_\mathcal{L}\\
	&~~~~~~+(-1)^{\alpha|x_1|}[s(x_1)+v_1,\mathcal{D}_l(s(x_2)+v_2)),s(x_3)+v_3]_\mathcal{L}\\
	&~~~~~~+(-1)^{\alpha(|x_1|+|x_2|)}[s(x_1)+v_1,s(x_2)+v_2,\mathcal{D}_l(s(x_3)+v_3)]_\mathcal{L},
\end{split}
\end{equation}
for $x_i\in \mathcal{Q}$ and $v_i\in \mathcal{P}$. Since $[\mathcal{P},\mathcal{P},\mathcal{L}]=0$, we get
\begin{equation*}
\begin{split}
		&[s(x_1)+v_1,s(x_2)+v_2,s(x_3)+v_3]_\mathcal{L}\\
	&=[s(x_1),s(x_2),s(x_3)]_\mathcal{L}+[s(x_1),s(x_2),v_3]_\mathcal{L}+[v_1,s(x_2),s(x_3)]_\mathcal{L}+[s(x_1),v_2,s(x_3)]_\mathcal{L}\\
	&=[s(x_1),s(x_2),s(x_3)]_\mathcal{L}+\Phi(x_1,x_2)(v_3)+(-1)^{|x_1|(|x_2|+|x_3|)}\Phi(x_2,x_3)(v_1)\\
	&~~~~~~+(-1)^{|x_3|(|x_1|+|x_2|)}\Phi(x_3,x_1)(v_2),  
\end{split}
\end{equation*}
and hence the left-hand side of Eq \ref{e45} is
\begin{equation*}
	\begin{split}
		&\mathcal{D}_l([s(x_1),s(x_2),s(x_3)]_\mathcal{L}+\Phi(x_1,x_2)(v_3)+(-1)^{|x_1|(|x_2|+|x_3|)}\Phi(x_2,x_3)(v_1)\\
		&~~~~~~+(-1)^{|x_3|(|x_1|+|x_2|)}\Phi(x_3,x_1)(v_2))\\
		&=\mathcal{D}_l(s([x_1,x_2,x_3]_\mathcal{Q})+\Omega(x_1,x_2,x_3)+\Phi(x_1,x_2)(v_3)+(-1)^{|x_1|(|x_2|+|x_3|)}\Phi(x_2,x_3)(v_1)\\
		&~~~~~~+(-1)^{|x_3|(|x_1|+|x_2|)}\Phi(x_3,x_1)(v_2)).
	\end{split}
\end{equation*}
Since $\mathcal{D}_l i=i \mathcal{D}_p$ where $i$ is the inclusion map, and $\Omega(x_1,x_2,x_3)$, $\Phi(x_1,x_2)(v_3)$, $\Phi(x_2,x_3)(v_1)$, $\Phi(x_3,x_1)(v_2)\in \mathcal{P}$, by the definition of $\mu$ as in Eq \ref{e43}, it follows that
\begin{equation}\label{e46}
	\begin{split}
		&s(\mathcal{D}_q([x_1,x_2,x_3]_\mathcal{Q})+\mu([x_1,x_2,x_3]_\mathcal{Q})+\mathcal{D}_p(\Omega(x_1,x_2,x_3))+\mathcal{D}_p(\Phi(x_1,x_2)(v_3))\\
		&~~~~~~+(-1)^{|x_1|(|x_2|+|x_3|)}\mathcal{D}_p(\Phi(x_2,x_3)(v_1))+(-1)^{|x_3|(|x_1|+|x_2|)}\mathcal{D}_p(\Phi(x_3,x_1)(v_2)))\\
		&=s([\mathcal{D}_q(x_1),x_2,x_3]_\mathcal{Q})+(-1)^{\alpha|x_1|}s([x_1,\mathcal{D}_q(x_2),x_3]_\mathcal{Q})\\
		&~~~~~~+(-1)^{\alpha(|x_2|+|x_3|)}s([x_1,x_2,\mathcal{D}_q(x_3)]_\mathcal{Q})+\mu([x_1,x_2,x_3]_\mathcal{Q})+\mathcal{D}_p(\Omega(x_1,x_2,x_3))\\
		&~~~~~~+\mathcal{D}_p(\Phi(x_1,x_2)(v_3))+(-1)^{|x_1|(|x_2|+|x_3|)}\mathcal{D}_p(\Phi(x_2,x_3)(v_1))\\
		&~~~~~~+(-1)^{|x_3|(|x_1|+|x_2|)}\mathcal{D}_p(\Phi(x_3,x_1)(v_2)).
	\end{split}
\end{equation}
Now we compute the right-hand side of Eq \ref{e45}. Since $\mathcal{D}
_l|_{\mathcal{P}}=\mathcal{D}_p$, 
\begin{equation*}
	\begin{split}
		\mathcal{D}_l(s(x_i)+v_i)&=\mathcal{D}_l(s(x_i))+\mathcal{D}_p(v_i)\\
		&=\mathcal{D}_l(s(x_i))-s(\mathcal{D}_q(x_i))+s(\mathcal{D}_q(x_i))+\mathcal{D}_p(v_i)\\
		&=s(\mathcal{D}_q(x_i))+\mu(x_i)+\mathcal{D}_p(v_i)\in s(\mathcal{Q})\oplus \mathcal{P}.
	\end{split}
\end{equation*}
By this and $[\mathcal{P},\mathcal{P},\mathcal{L}]=0$, the right-hand side of Eq \ref{e45} is
\begin{equation}\label{e47}
	\begin{split}
		&[s(\mathcal{D}_q(x_1))+\mu(x_1)+\mathcal{D}_p(v_1),s(x_2)+v_2,s(x_3)+v_3]_\mathcal{L}\\
		&~~~~~~+(-1)^{|\alpha||x_1|}[s(x_1)+v_1,s(\mathcal{D}_q(x_2))+\mu(x_2)+\mathcal{D}_p(v_2),s(x_3)+v_3]_\mathcal{L}\\
		&~~~~~~+(-1)^{|\alpha|(|x_1|+|x_2|)}[s(x_1)+v_1,s(x_2)+v_2,s(\mathcal{D}_q(x_3))+\mu(x_3)+\mathcal{D}_p(v_3)]_\mathcal{L}\\
		&=[s(\mathcal{D}_q(x_1)),s(x_2),s(x_3)]_\mathcal{L}+[s(\mathcal{D}_q(x_1)),s(x_2),v_3]_\mathcal{L}+[s(\mathcal{D}_q(x_1)),v_2,s(x_3)]_\mathcal{L}\\
		&~~~~~~+[\mu(x_1),s(x_2),s(x_3)]_\mathcal{L}+[\mathcal{D}_p(v_1),s(x_2),s(x_3)]_\mathcal{L}+(-1)^{\alpha|x_1|}([s(x_1),s(\mathcal{D}_q(x_2)),s(x_3)]_\mathcal{L}\\
		&~~~~~~+[s(x_1),s(\mathcal{D}_q(x_2)),v_3]_\mathcal{L}+[s(x_1),\mu(x_2),s(x_3)]_\mathcal{L}+[s(x_1),\mathcal{D}_p(v_2),s(x_3)]_\mathcal{L}\\
		&~~~~~~+[v_1,s(\mathcal{D}_q(x_2)),s(x_3)]_\mathcal{L})+(-1)^{\alpha(|x_1|+|x_2|)}([s(x_1),s(x_2),s(\mathcal{D}_q(x_3))]_\mathcal{L}\\
		&~~~~~~+[v_1,s(x_2),s(\mathcal{D}_q(x_3))]_\mathcal{L}+[s(x_1),s(x_2),\mu(x_3)]_\mathcal{L}+[s(x_1),s(x_2),\mathcal{D}_p(v_3)]_\mathcal{L}\\
		&~~~~~~+[s(x_1),v_2,s(\mathcal{D}_q(x_3))]_\mathcal{L}).
	\end{split}
\end{equation}
By Eqs \ref{e46} and \ref{e47} it follows that
\begin{equation*}
	\begin{split}
	&s([\mathcal{D}_q(x_1),x_2,x_3]_\mathcal{Q})+(-1)^{\alpha|x_1|}s([x_1,\mathcal{D}_q(x_2),x_3]_\mathcal{Q})+(-1)^{\alpha(|x_1|+|x_2|)}s([x_1,x_2,\mathcal{D}_q(x_3)]_\mathcal{Q})\\
	&~~~~~~+\mu([x_1,x_2,x_3]_\mathcal{Q})+\mathcal{D}_p(\Omega(x_1,x_2,x_3))+\mathcal{D}_p(\Phi(x_1,x_2)(v_3))\\&~~~~~~+(-1)^{|x_1|(|x_2|+|x_3|)}\mathcal{D}_p(\Phi(x_2,x_3)(v_1))+(-1)^{|x_3|(|x_1|+|x_2|)}\mathcal{D}_p(\Phi(x_3,x_1)(v_2))\\
	&=[s(\mathcal{D}_q(x_1)),s(x_2),s(x_3)]_\mathcal{L}+\Phi(\mathcal{D}_q(x_1),x_2)(v_3)+(-1)^{|x_3|(\alpha+|x_1|+|x_2|)}\Phi(x_3,\mathcal{D}_q(x_1))(v_2)\\
	&~~~~~~~+(-1)^{|x_1|(|x_2|+|x_3|)}\Phi(x_2,x_3)(\mu(x_1))+(-1)^{(\alpha+|x_1|)(|x_2|+|x_3|)}\Phi(x_2,x_3)(\mathcal{D}_p(v_1))\\
	&~~~~~~+(-1)^{\alpha|x_1|}([s(x_1),s(\mathcal{D}_q(x_2)),s(x_3)]_\mathcal{L}+(-1)^{\alpha|x_1|}\Phi(x_1,\mathcal{D}_q(x_2))(v_3)\\
	&~~~~~~+(-1)^{\alpha(|x_1|+|x_3|)+|x_1|(|x_2|+|x_3|)}\Phi(x_3,x_1)(\mathcal{D}_p(v_2))+(-1)^{|x_3|(|x_1|+|x_2|)}\Phi(x_3,x_1)(\mu(x_2))\\
	&~~~~~~+(-1)^{|x_1|(|x_2|+|x_3|)}\Phi(\mathcal{D}_q(x_2),x_3)(v_1))+(-1)^{\alpha(|x_1|+|x_2|)}([s(x_1),s(x_2),s(\mathcal{D}_q(x_3))]_\mathcal{L}\\
	&~~~~~~+(-1)^{\alpha(|x_1|+|x_2|)}\Phi(x_1,x_2)(\mu(x_3))+(-1)^{\alpha(|x_1|+|x_2|)}\Phi(x_1,x_2)(\mathcal{D}_p(v_3))\\&~~~~~~+(-1)^{|x_3|(|x_1|+|x_2|)}\Phi(\mathcal{D}_q(x_3),x_1)(v_2)+(-1)^{\alpha(|x_1|+|x_2|)+|x_1|(\alpha+|x_2|+|x_3|)}\Phi(x_2,\mathcal{D}_q(x_3))(v_1)).
\end{split}
\end{equation*}
Then we get
\begin{equation*}
	\begin{split}
		0&=-\Omega(\mathcal{D}_q(x_1),x_2,x_3)-(-1)^{\alpha|x_1|}\Omega(x_1,\mathcal{D}_q(x_2),x_3)-(-1)^{\alpha(|x_1|+|x_2|)}\Omega(x_1,x_2,\mathcal{D}_q(x_3))\\
		&~~~~~~+\mathcal{D}_p(\Omega(x_1,x_2,x_3))+(-1)^{|x_1|(|x_2|+|x_3|)}\Phi(x_2,x_3)(\mu(x_1))+(-1)^{|x_3|(|x_1|+|x_2|)}\Phi(x_3,x_1)(\mu(x_2))\\&~~~~~~+(-1)^{\alpha(|x_1|+|x_2|)}\Phi(x_1,x_2)(\mu(x_3))+\mu([x_1,x_2,x_3])+(\mathcal{D}_p\Phi(x_1,x_2)\\&~~~~~~-(-1)^{\alpha(|x_1|+|x_2|)}\Phi(x_1,x_2)\mathcal{D}_p-\Phi(\mathcal{D}_q(x_1),x_2)-(-1)^{\alpha|x_1|}\Phi(x_1,\mathcal{D}_q(x_2)))(v_3)\\&~~~~~~+(-1)^{|x_1|(|x_2|+|x_3|)}(\mathcal{D}_p\Phi(x_2,x_3)-(-1)^{\alpha(|x_2|+|x_3|)}\Phi(x_2,x_3)\mathcal{D}_p-\Phi(\mathcal{D}_q(x_2),x_3)\\&~~~~~~-(-1)^{\alpha|x_2|}\Phi(x_2,\mathcal{D}_q(x_3)))(v_1)+(-1)^{|x_3|(|x_1|+|x_2|)}(\mathcal{D}_p\Phi(x_3,x_1)\\&~~~~~~-(-1)^{\alpha(|x_3|+|x_1|)}\Phi(x_3,x_1)\mathcal{D}_p-\Phi(\mathcal{D}_q(x_3),x_1)-(-1)^{\alpha|x_1|}\Phi(x_3,\mathcal{D}_q(x_1)))(v_2).
\end{split}
\end{equation*}
Since $(\mathcal{D}_p,\mathcal{D}_q)$ is compatible with respect to $\Phi$,

\begin{equation*}
	\begin{split}
		(\mathcal{D}_p\Phi(x_1,x_2)&-(-1)^{\alpha(|x_1|+|x_2|)}\Phi(x_1,x_2)\mathcal{D}_p-\Phi(\mathcal{D}_q(x_1),x_2)\\&-(-1)^{\alpha|x_1|}\Phi(x_1,\mathcal{D}_q(x_2)))(v_3)=0,		
	\end{split}
\end{equation*}

\begin{equation*}
	\begin{split}
		 (\mathcal{D}_p\Phi(x_2,x_3)&-(-1)^{\alpha(|x_2|+|x_3|)}\Phi(x_2,x_3)\mathcal{D}_p-\Phi(\mathcal{D}_q(x_2),x_3)\\&-(-1)^{\alpha|x_2|}\Phi(x_2,\mathcal{D}_q(x_3)))(v_1)=0,
	\end{split}
\end{equation*}

\begin{equation*}
	\begin{split}
	 (\mathcal{D}_p\Phi(x_3,x_1)&-(-1)^{\alpha(|x_3|+|x_1|)}\Phi(x_3,x_1)\mathcal{D}_p-\Phi(\mathcal{D}_q(x_3),x_1)\\&-(-1)^{\alpha|x_1|}\Phi(x_3,\mathcal{D}_q(x_1)))(v_2)=0.
		\end{split}
\end{equation*}
Thus we have
\begin{equation*}
	\begin{split}
		&-\Omega(\mathcal{D}_q(x_1),x_2,x_3)-(-1)^{\alpha|x_1|}\Omega(x_1,\mathcal{D}_q(x_2),x_3)-(-1)^{\alpha(|x_1|+|x_2|)}\Omega(x_1,x_2,\mathcal{D}_q(x_3))\\
		&+\mathcal{D}_p(\Omega(x_1,x_2,x_3))+(-1)^{|x_1|(|x_2|+|x_3|)}\Phi(x_2,x_3)(\mu(x_1))+(-1)^{|x_3|(|x_1|+|x_2|)}\Phi(x_3,x_1)(\mu(x_2))\\&+(-1)^{\alpha(|x_1|+|x_2|)}\Phi(x_1,x_2)(\mu(x_3))+\mu([x_1,x_2,x_3])=0,
	\end{split}
\end{equation*}
since $|\mathcal{D}_p|=|\mathcal{D}_q|=|\mathcal{D}_l|=|\alpha|=0$, hence we have $Ob^\mathcal{L}_{(\mathcal{D}_p,\mathcal{D}_q)}(x_1,x_2,x_3)=(\delta_{\Phi}(\mu))(x_1,x_2,x_3)$ due to Eqs \ref{e32} and \ref{e25}.\\
To prove the converse part refer \cite{xu2018}.
\end{proof}

\end{document}